\documentclass[a4paper,11pt,DIV=11,% decrease borders by 2 (std 10 for 11pt)
abstract=on% return to old style centered abstract
]{scrartcl}
\usepackage{mathtools}
\mathtoolsset{showonlyrefs}
\usepackage{amsmath, amsthm, amssymb}
\usepackage{fontenc}
\usepackage[utf8]{inputenc}
\usepackage{graphicx}
\usepackage{xargs}
\usepackage[subrefformat=parens,labelformat=parens]{subcaption}
\usepackage[noend]{algpseudocode}
\usepackage{enumitem,listings,microtype,tikz,pgfplots}
\usepackage{hyperref}
\usepackage{environ,ifthen}
\usepackage{multirow}
\usepackage{verbatim}
\usepackage{bbold}
\usepackage{rotating}

\newtheorem{Theorem}{Theorem}[section]

\newtheorem{Lemma}[Theorem]{Lemma}

\newtheorem{Corollary}[Theorem]{Corollary}
\newtheorem{Assumption}[Theorem]{Assumption}
\setkomafont{sectioning}{\rmfamily\bfseries}
\setkomafont{title}{\rmfamily}
\setkomafont{descriptionlabel}{\rmfamily\bfseries}

\usetikzlibrary{spy}
\usetikzlibrary{calc}
\usetikzlibrary{external}

\usepackage{algorithm}
\usepackage{booktabs}

\usepackage[noend]{algpseudocode}

\usepackage{setspace}			% 1,5-facher Zeilenabstand
\usepackage{flafter}			% Floats werden nicht vor ihrer Definitionsstelle gesetzt

\usepackage{appendix}
\usepackage{hyperref}	
\hypersetup{colorlinks=true,linkcolor=black,urlcolor=blue,citecolor=black}		% für Hyperlinks in PDF
\usepackage{placeins} 		%	 für FloatBarrier, das die Ausgabe der Grafiken bis zu diesem Punkt erzwingt

\usepackage{listings}
\usepackage{empheq}

\newcommand{\N}{\mathbb{N}}

\newcommand{\R}{\mathbb{R}}

\newcommand{\SP}{\mathbb{S}}

\newcommand{\NN}{\mathcal{N}}

\newcommand{\GG}{\mathcal{G}}

\newcommand{\HH}{\mathcal{H}}

\newcommand{\LL}{\mathcal{L}}

\newcommand{\E}{\mathbb{E}}

\newcommand{\abs}[1]{\left| #1 \right|}

\newcommand{\dx}[1][x]{\,\mathrm{d}#1}

\newcommand{\tT}{\mathrm{T}}

\newcommand{\e}{\mathrm{e}}

\newcommand{\n}{\phantom{0}}

\def\tT{{\mbox{\tiny{T}}}}

\newcommand\inner[2]{\left\langle #1, #2 \right\rangle}

\newcommand\eps{\varepsilon}
\renewcommand\theta{\vartheta}

\DeclareMathOperator*{\argmin}{argmin}

\DeclareMathOperator{\tr}{tr}

\DeclareMathOperator{\Var}{Var}
\DeclareMathOperator{\Cov}{Cov}

\DeclareMathOperator{\SPD}{SPD}

\allowdisplaybreaks[4]
\begin{document}

\title{
Alternatives to the EM Algorithm for ML-Estimation of Location, Scatter
Matrix and Degree of Freedom of the Student-$t$ Distribution 
}
\author{Marzieh Hasannasab\footnotemark[2]
\and
Johannes Hertrich\footnotemark[1]
\and
Friederike Laus\footnotemark[1]
\and
Gabriele Steidl\footnotemark[2]
}

\maketitle

\footnotetext[1]{Department of Mathematics,
Technische Universit\"at Kaiserslautern,
Paul-Ehrlich-Str.~31, D-67663 Kaiserslautern, Germany,
\{hasannas, jhertric, friederike.laus\}@mathematik.uni-kl.de.
} 

\footnotetext[2]{Institute of Mathematics,
TU Berlin,
Straße des 17. Juni 136, 
 D-10623 Berlin, Germany,
steidl@math.tu-berlin.de.
}

\begin{abstract}
In this paper, we consider maximum likelihood estimations of the degree of freedom parameter $\nu$, 
the location parameter $\mu$  and the scatter matrix $\Sigma$ 
of the multivariate Student-$t$ distribution.
In particular, we are interested in estimating the degree of freedom parameter $\nu$ that determines 
the tails of the corresponding probability density function and was rarely considered in detail in the literature so far.
We prove that under certain assumptions a minimizer of the negative log-likelihood function exists, where we have to take special care of the case
$\nu \rightarrow \infty$, for which the Student-$t$ distribution approaches the Gaussian distribution.
As alternatives to the classical EM algorithm we propose three other algorithms which cannot be interpreted
as EM algorithm. For fixed $\nu$, the first algorithm is an accelerated EM algorithm known from the literature.
However, since we do not fix $\nu$, we cannot apply standard convergence results for the EM algorithm.
The other two algorithms differ from this algorithm in the iteration step for $\nu$. 
We show how the objective function   behaves for the different updates of $\nu$
and prove for all three algorithms that it  decreases in each iteration step.
We compare the algorithms as well as some accelerated versions 
by numerical simulation and apply one of them for estimating the degree of freedom parameter
in images corrupted by Student-$t$ noise.   
\end{abstract}

%------------------------------------------
\section{Introduction} \label{sec:intro}
%------------------------------------------
The motivation for this work arises from certain tasks in image processing, where the robustness of methods plays an important role.
In this context, the Student-$t$ distribution and the closely related Student-$t$ mixture models became popular
in various image processing tasks.  
In~\cite{VS14} it has been shown that Student-$t$ mixture models are superior 
to Gaussian mixture models for mode\-ling image patches and the authors proposed an application in image compression. 
Image denoising based on Student-$t$ models was addressed in~\cite{LS2019} and image
deblurring in~\cite{DHWMZ2019,YYG2018}.
Further applications include robust image segmentation~\cite{BM18,NW12,SNG07} as well as robust registration~\cite{GNL09,ZZDZC14}.

In one dimension and for $\nu=1$, the Student-$t$ distribution coincides with the one-dimensional Cauchy distribution.
This distribution has been proposed to model a very impulsive noise behavior and
one of the first papers which suggested a variational approach in connection with wavelet shrinkage
for denoising of images corrupted by Cauchy noise was 
\cite{ALP2002}. A variational method consisting of a data term 
that resembles the noise statistics and a total variation regularization term has been introduced in~\cite{MDHY18,SDZ15}. 
Based on an ML approach the authors of~\cite{LPS18} introduced a so-called
generalized  myriad filter that estimates both the location and the scale parameter of the Cauchy distribution.
They used the filter in a nonlocal denoising approach,
where for each pixel of the image they chose as samples of the distribution those pixels  having a similar neighborhood 
and replaced the initial pixel by its filtered version.
We also want to mention that a unified framework for images corrupted by white noise 
that can handle (range constrained) Cauchy noise as well was suggested in~\cite{LMSS2018}.

In contrast to the above pixelwise replacement, the state-of-the-art algorithm of Lebrun et al.~\cite{LBM13}
for denoising images corrupted by white Gaussian noise
restores the image  patchwise  based on a maximum a posteriori approach.
In the Gaussian setting, their approach is equivalent to
minimum mean square error estimation, and more general, the resulting estimator can be
seen as a particular instance of a best linear unbiased estimator (BLUE). 
For denoising images corrupted by additive Cauchy noise, a similar approach was addressed in \cite{LS2019} based 
on ML estimation for the family of Student-t distributions, of which the
Cauchy distribution forms a special case.
The authors call this approach generalized multivariate myriad filter.

However, all these approaches assume that the degree of freedom parameter $\nu$ of the Student-$t$
distribution is known, which might not be the case in practice. In this paper we  consider the estimation of the degree of freedom parameter based on an ML approach. 
In contrast to maximum likelihood estimators of the location and/or scatter parameter(s) $\mu$ and $\Sigma$, 
to the best of our knowledge the question of existence of a joint maximum likelihood estimator has not been analyzed before and in this paper we provide first results in this direction.
Usually the likelihood function  of the Student-$t$ distributions and mixture models  
are minimized using the EM algorithm derived e.g.\ in~\cite{LLT89,McLK1997,MP98,PM00}.
For fixed $\nu$, there exists an accelerated EM algorithm~\cite{KTV94,MVD97,vanDyk1995} 
which appears to be more efficient than the classical one for smaller parameters $\nu$. 
We examine the convergence of the accelerated version if also the
degree of freedom parameter $\nu$ has to be estimated. Also for unknown degrees of freedom, there exist an accelerated version of the EM algorithm, 
the so-called ECME algorithm~\cite{LR95} which differs from our algorithm.
Further, we propose two  modifications of the $\nu$ iteration step which lead to efficient algorithms
for  a wide range of parameters $\nu$. Finally, we address further accelerations of our algorithms by
the squared iterative methods (SQUAREM) \cite{VR2008} and the
damped Anderson acceleration with restarts and $\epsilon$-monotonicity (DAAREM) \cite{HV2019}.

The paper is organized as follows: 
In Section \ref{sec:ML} we introduce the Student-$t$ distribution, the negative $\log$-likelihood function $L$
and their derivatives.
The question of the existence of a minimizer of $L$ is addressed in Section \ref{sec:exits_nu_scatter}. 
Section \ref{sec:zero_F} deals with the solution of the equation arising when setting the gradient of $L$
with respect to $\nu$ to zero. The results of this section will be important for the convergence
consideration of our algorithms in the 
Section \ref{sec:algs}. We propose three alternatives of the classical EM algorithm and prove 
that the objective function $L$ decreases for the iterates produced by these algorithms.
Finally, we provide two kinds of numerical results in Section \ref{sec:algs}.
First, we compare the different algorithms by numerical examples which indicate that the new $\nu$ iterations
are very efficient for estimating $\nu$ of different magnitudes.
Second, we come back to the original motivation of this paper and  estimate the degree of freedom parameter $\nu$
from images corrupted by one-dimensional Student-$t$ noise.

%----------------------------------------------------
\section{Likelihood of the Multivariate Student-$t$ Distribution} \label{sec:ML}
%----------------------------------------------------
The density function of the 
$d$-dimensional Student-$t$ distribution $T_\nu(\mu,\Sigma)$ with 
$\nu>0$ degrees of freedom, \emph{location} paramter $\mu\in \R^d$ and symmetric,  positive definite \emph{scatter matrix} $\Sigma\in \SPD(d)$ 
is given by
\begin{equation}\label{pdf}
p(x|\nu,\mu,\Sigma)  = 
\frac{\Gamma\left(\frac{d+\nu}{2}\right)}{\Gamma\left(\frac{\nu}{2}\right)\, \nu^{\frac{d}{2}} \, \pi^{\frac{d}{2}} \,
{\abs{\Sigma}}^{\frac{1}{2}}} \, \frac{1}{\left(1 +\frac1\nu(x-\mu)^\tT \Sigma^{-1}(x-\mu) \right)^{\frac{d+\nu}{2}}},
\end{equation}
with the \emph{Gamma function}
$
\Gamma(s) \coloneqq\int_0^\infty t^{s-1}\e^{-t}\dx[t] 
$.
The expectation of the Student-$t$ distribution is $\E(X) = \mu$ for $\nu > 1$ 
and the covariance matrix is given by $\Cov(X) =\frac{\nu }{\nu-2} \Sigma$ for $\nu > 2$, 
otherwise the quantities are undefined. 
The smaller the value of $\nu$, the heavier are the tails of the $T_\nu(\mu,\Sigma)$ distribution.
For $\nu \to \infty$, 
the Student-$t$ distribution $T_\nu(\mu,\Sigma)$ converges to the normal distribution $\NN(\mu,\Sigma)$ and for $\nu = 0$
it is related to the projected normal distribution on the sphere $\SP^{d-1}\subset\R^d$.
Figure~\ref{Fig:different_nu} 
illustrates this behavior for the one-dimensional standard Student-$t$ distribution. 

\begin{figure}[thb]
\centering  
\centering  
{\includegraphics[width=0.4\textwidth]{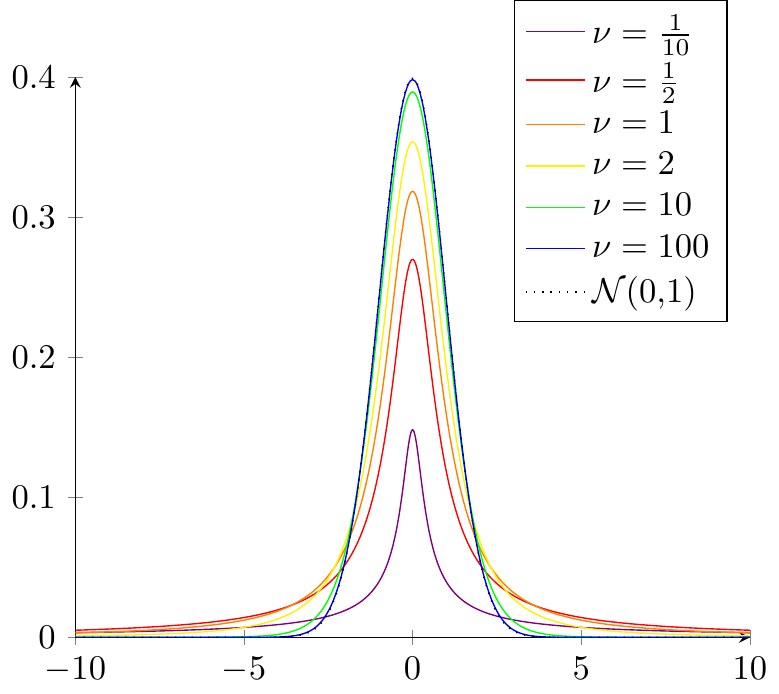}}
\caption{Standard Student-$t$ distribution $T_\nu(0,1)$ 
for different values of $\nu$ in comparison with the standard normal distribution $\NN(0,1)$.}\label{Fig:different_nu}
\end{figure}

As the normal distribution, the $d$-dimensional Student-$t$ distribution belongs to the class of \emph{elliptically symmetric distributions}.
These distributions are stable under linear transforms in the following sense:  Let $X\sim T_\nu(\mu,\Sigma)$  and $A\in \R^{d\times d}$ be an invertible matrix and let $b\in \R^d$. 
Then $AX + b\sim T_\nu(A\mu + b, A\Sigma A^\tT)$. Furthermore, the Student-$t$ distribution $T_\nu(\mu,\Sigma)$ admits the following \emph{stochastic representation}, which can be used to generate samples from $T_\nu(\mu,\Sigma)$ based on samples from the multivariate standard normal distribution $\mathcal{N}(0,I)$ and the Gamma distribution $\Gamma\bigl(\tfrac{\nu}{2},\tfrac{\nu}{2}\bigr)$: Let $Z\sim \mathcal{N}(0,I)$ and $Y\sim \Gamma\bigl(\tfrac{\nu}{2},\tfrac{\nu}{2}\bigr)$ be independent, then
\begin{equation} X = \mu + \frac{\Sigma^{\frac{1}{2}}Z}{\sqrt{Y}}\sim  T_\nu(\mu,\Sigma).\label{stochastic_representation}
\end{equation}

For i.i.d.\ samples $x_i \in \mathbb R^d$, $i=1,\ldots,n$, 
the likelihood function of the Student-$t$ distribution $T_\nu(\mu,\Sigma)$ is given by
\begin{equation*}
\LL(\nu,\mu,\Sigma|x_1,\ldots,x_n)
= \frac{\Gamma\left(\frac{d+\nu}{2}\right)^n}{\Gamma\left(\frac{\nu}{2}\right)^n(\pi \nu)^{\frac{nd}{2}}\abs{\Sigma}^{\frac{n}{2}} } 
\prod_{i=1}^n \frac{1}{\bigl(1+\frac{1}{\nu}(x_i-\mu)^\tT \Sigma^{-1} (x_i-\mu)\bigr)^{\frac{d+\nu}{2}}},
\end{equation*}
and the log-likelihood function by
\begin{align}
\ell(\nu,\mu,\Sigma|x_1,\ldots,x_n) 
=&  
\, n \, \log\Bigl(\Gamma\left(	\tfrac{d+\nu}{2}\right)\Bigr) 
- n \log \Bigl( \Gamma\left(\tfrac{\nu}{2}\right)\Bigr)-\tfrac{nd}{2}\log(\pi\nu) \\
&- \frac{n}{2}\log \abs{\Sigma} - \tfrac{d+\nu}{2} \sum_{i=1}^n \log\left(1+\frac{1}{\nu}(x_i-\mu)^\tT \Sigma^{-1} (x_i-\mu) \right).
\end{align}
In the following, we are interested in the negative log-likelihood function, which up to the factor $\frac{2}{n}$ and weights $w_i = \frac{1}{n}$ reads as  
\begin{align} \label{ML}
L(\nu,\mu,\Sigma)
&= -2\log\Bigl(\Gamma\left(	\tfrac{d+\nu}{2}\right)\Bigr)+ 2 \log\Bigl(\Gamma\left(	\tfrac{\nu}{2}\right)\Bigr) - \nu \log(\nu) \\
 &\quad +  (d + \nu)\sum_{i=1}^n w_i \log\left(\nu + (x_i-\mu)^\tT \Sigma^{-1} (x_i-\mu) \right)+ \log \abs{\Sigma}.
\end{align}
In this paper, we allow for arbitrary weights from the open probability simplex
$
\mathring \Delta_n \coloneqq  \big\{w = (w_1,\ldots,w_n) \in \mathbb R_{>0}^n: \sum_{i=1}^n w_i = 1 \big\}
$.
In this way, we might express different levels of confidence in single samples or handle the occurrence of multiple samples.
Using 
$
\frac{\partial \log(\abs{X})}{\partial X} = X^{-1}
$ and
$
\frac{\partial a^\tT X^{-1}b }{\partial X} =- {(X^{-\tT})}a b^\tT {(X^{-\tT})},
$
see~\cite{PP08}, the derivatives of $L$ with respect to $\mu$, $\Sigma$ and $\nu$ are given by
\begin{align*}
\frac{\partial L}{\partial \mu}(\nu,\mu,\Sigma) 
& = -2(d+\nu )\sum_{i=1}^n w_i \frac{ \Sigma^{-1}(x_i-\mu)}{\nu + (x_i-\mu)^\tT \Sigma^{-1} (x_i-\mu)},\\
\frac{\partial L}{\partial \Sigma}(\nu,\mu,\Sigma)	
& = - (d+\nu ) \sum_{i=1}^n w_i \frac{ \Sigma^{-1}(x_i-\mu)(x_i-\mu)^\tT \Sigma^{-1} }{\nu + (x_i-\mu)^\tT \Sigma^{-1} (x_i-\mu)}+\Sigma^{-1},\\
\frac{\partial L}{\partial \nu}(\nu,\mu,\Sigma ) 
& = 
\phi\left(\frac{\nu}{2}\right) - \phi \left(\frac{\nu + d}{2}\right) + \sum_{i=1}^n w_i \left( \frac{\nu + d}{\nu +  (x_i-\mu)^\tT \Sigma^{-1} (x_i-\mu)}\right.\\
 & \quad \left. - \log\left(\frac{\nu + d}{\nu +  (x_i-\mu)^\tT \Sigma^{-1} (x_i-\mu)} \right) - 1\right),
\end{align*}
with 
$$\phi(x) \coloneqq \psi(x) - \log (x), \qquad x >0$$ 
and the \emph{digamma function}
$$
\psi(x) = \frac{\mathrm{d}}{\mathrm{d}x}\log\left(\Gamma(x)\right) = \frac{\Gamma'(x)}{\Gamma(x)}.
$$
Setting the derivatives to zero results in the equations
\begin{align}
0 &= \sum_{i=1}^n w_i \frac{x_i-\mu}{\nu+(x_i-\mu)^\tT \Sigma^{-1} (x_i-\mu)},\label{mult_cond_a}\\
I &=	(d+\nu)\sum_{i=1}^n w_i \frac{\Sigma^{-\frac{1}{2}}(x_i-\mu)(x_i-\mu)^\tT {\Sigma^{-\frac{1}{2}}} }{\nu+(x_i-\mu)^\tT \Sigma^{-1} (x_i-\mu)} ,\label{mult_cond_S}\\
0 &= F\left(\frac{\nu }{2} \right) \coloneqq \phi\left(\frac{\nu }{2}\right) - \phi\left(\frac{\nu +d}{2}\right) \\
 &\quad + \sum_{i=1}^n w_i \left( \frac{\nu + d}{\nu +  (x_i-\mu)^\tT \Sigma^{-1} (x_i-\mu)}- \log\left(\frac{\nu + d}{\nu +  (x_i-\mu)^\tT \Sigma^{-1} (x_i-\mu)} \right) - 1\right). \label{ML_nu}
\end{align}
Computing the trace of both sides of~\eqref{mult_cond_S} and using the linearity and permutation invariance of the trace operator we obtain
\begin{align}
d& = \tr(I) 
=
(d+\nu)\sum_{i=1}^n w_i \frac{\tr\bigl(\Sigma^{-\frac{1}{2}}(x_i-\mu)(x_i-\mu)^\tT {\Sigma^{-\frac{1}{2}}}\bigr)}{\nu+(x_i-\mu)^\tT \Sigma^{-1} (x_i-\mu)}  \\
 &= (d+\nu)\sum_{i=1}^n w_i \frac{(x_i-\mu)^\tT \Sigma^{-1} (x_i-\mu)}{\nu+(x_i-\mu)^\tT \Sigma^{-1} (x_i-\mu)},
\end{align}
which yields
\begin{equation} \label{trace_1}
	1= (d+\nu)	\sum_{i=1}^n w_i \frac{1}{\nu+(x_i-\mu)^\tT \Sigma^{-1} (x_i-\mu)}.
\end{equation}
We are interested in critical points of the negative log-likelihood function $L$, i.e. in solutions $(\mu,\Sigma,\nu)$ of 
\eqref{mult_cond_a} - \eqref{ML_nu}, and in particular in minimizers of $L$.

%-------------------------------------------------------------
\section{Existence of Critical Points}\label{sec:exits_nu_scatter}
%-------------------------------------------------------------
In this section, we examine whether the negative log-likelihood function $L$ has a minimizer, where we restrict our attention to the case $\mu = 0$. For an approach how to extend the results to arbitrary $\mu$ for fixed $\nu$  we refer to \cite{LS2019}. To the best of our knowledge, this is the first work that provides results in this direction. The question of existence is, however, crucial in the context of ML estimation, since it lays the foundation for  any convergence result for the EM algorithm or its variants. In fact, the authors of~\cite{LLT89} observed the   divergence of the EM algorithm  in some of their numerical experiments, which is in accordance with our observations.

For \emph{fixed} $\nu >0$, it is known that there exists a unique solution of \eqref{mult_cond_S} 
and for $\nu = 0$ that there exists solutions of \eqref{mult_cond_S}
which differ only by a multiplicative positive constant, see, e.g. \cite{LS2019}.
In contrast, if we do not fix $\nu$, we have roughly to distinguish between the two cases 
that the samples tend to come from a Gaussian distribution, i.e.\ $\nu\to\infty$, or not.
The results are presented in Theorem \ref{thm:Existence of Scatter}.

We make the following general assumption:

\begin{Assumption}\label{Ass:lin_ind} 
Any subset of $\le d$ samples  $x_i$, $i \in \{1,\ldots,n\}$ is linearly independent and
$\max\{w_i:i=1,\ldots,n\}<\frac{1}{d }$.
\end{Assumption}
 
For $\mu = 0$, the negative log-likelihood function becomes
\begin{align}
L(\nu,\Sigma) 
&\coloneqq  -2\log\left(\Gamma\left(\frac{d+\nu}{2}\right)\right)+2\log\left(\Gamma\left(\frac\nu2\right)\right)-\nu\log(\nu)\\
&\quad +(d+\nu)\sum_{i=1}^n w_i \log\left(\nu+ x_i^\tT\Sigma^{-1}x_i\right)+\log(\abs{\Sigma})\\
&=-2\log\left(\Gamma\left(\frac{d+\nu}{2}\right)\right)+2\log\left(\Gamma\left(\frac\nu2\right)\right)-\nu\log(\nu)\\
&\quad +(d+\nu)\log(\nu)+(d+\nu)\sum_{i=1}^nw_i\log\left(1+ \frac1\nu x_i^\tT\Sigma^{-1}x_i\right)+\log(\abs{\Sigma}).
\end{align}
Further, for a fixed $\nu>0$, set
\begin{equation*}
	L_\nu(\Sigma)  \coloneqq (d+\nu)\sum_{i=1}^n w_i \log\left(\nu+ x_i^\tT\Sigma^{-1}x_i\right)+\log(\abs{\Sigma}).
\end{equation*}

To prove the next existence theorem we will need two lemmas, whose proofs are given in the appendix.

\begin{Theorem} \label{thm:Existence of Scatter}
Let $x_i \in \mathbb R^d$, $i=1,\ldots,n$ and $w \in \mathring \Delta_n$ fulfill Assumption \ref{Ass:lin_ind}. 
Then exactly one of the following statements holds:
\begin{enumerate}
\item[(i)] There exists a minimizing sequence $(\nu_r,\Sigma_r)_r$ of $L$, 
such that $\{\nu_r:r\in\mathbb N\}$ has a finite cluster point. Then we have
$\argmin_{(\nu,\Sigma)\in\R_{>0}\times\mathrm{SPD}(d)} L(\nu,\Sigma)\neq\emptyset$ 
and every 
$(\hat \nu,\hat \Sigma)\in\argmin_{(\nu,\Sigma)\in\R_{>0}\times\mathrm{SPD}(d)}L(\nu,\Sigma)$ 
is a critical point of $L$.
\item[(ii)] 
For every minimizing sequence $(\nu_r,\Sigma_r)_r$ of $L(\nu,\Sigma)$ we have
$\lim\limits_{r\to\infty} \nu_r=\infty$. Then
$(\Sigma_r)_r$ converges to the maximum likelihood estimator $\hat\Sigma=\sum_{i=1}^n w_ix_ix_i^\tT$ 
of the normal distribution $\mathcal{N}(0,\Sigma)$.
\end{enumerate}
\end{Theorem}

\begin{proof}
\textbf{Case 1:} Assume that there exists a minimizing sequence $(\nu_r,\Sigma_r)_r$ of $L$, such that $(\nu_r)_r$ has a bounded subsequence. 
In particular, using Lemma \ref{lem:rike}, we have that $(\nu_r)_r$ has a cluster point $\nu^* >0$ 
and a subsequence $(\nu_{r_k})_k$ converging to $\nu^*$. 
Clearly, the sequence $(\nu_{r_k},\Sigma_{r_k})_k$ is again a minimizing sequence so that we skip the second index in the following. 
By Lemma \ref{lem:lik}, the set $\overline{\{\Sigma_r:r\in\mathbb N\}}$ 
is a compact subset of $\mathrm{SPD}(d)$. 
Therefore there exists a subsequence $(\Sigma_{r_k})_k$ 
which converges to some $\Sigma^*\in\mathrm{SPD}(d)$. Now we have by continuity of $L(\nu,\Sigma)$ that
$$
L(\nu^*,\Sigma^*)=\lim\limits_{k\to\infty}L(\nu_{r_k},\Sigma_{r_k})=\min_{(\nu,\Sigma)\in\R_{>0}\times\mathrm{SPD}(d)} L(\nu,\Sigma).
$$
\textbf{Case 2:} Assume that for every minimizing sequence $(\nu_r,\Sigma_r)_r$ it holds that $\nu_r\to\infty$ as $r\to \infty$. 
We rewrite the likelihood function as
\begin{align}
L(\nu,\Sigma) 
&=
2\log \left( \frac{\Gamma\left(\frac\nu2\right)\frac\nu2^\frac{d}{2} }{ \Gamma\left(\frac{d+\nu}{2} \right)} \right)
+d \log(2)
+(d+\nu) \sum_{i=1}^n w_i
\log \left(1+\frac1\nu x_i^\tT \Sigma^{-1}x_i\right)+\log(\abs{\Sigma}).
\end{align}
Since
\[\lim_{\nu \rightarrow \infty}  \frac{\Gamma\left(\frac\nu2\right)\frac\nu2^\frac{d}{2} }{ \Gamma\left(\frac{d+\nu}{2} \right)}=1,\]
we obtain
\begin{equation}\label{eq:asym_likelihood}
\lim\limits_{r\to\infty}L(\nu_r,\Sigma_r)=
d\log(2)+ \lim_{\nu_r \rightarrow \infty} (d+\nu_r)\sum_{i=1}^nw_i\log\left(1+\frac1{\nu_r} x_i^\tT \Sigma_r^{-1}x_i\right)+\log(\abs{\Sigma_r}).
\end{equation}
Next we show by contradiction that $\overline{\{\Sigma_r:r\in\mathbb N\}}$ is in $\mathrm{SPD}(d)$ and bounded:
Denote the eigenvalues of $\Sigma_r$ by $\lambda_{r1}\geq\cdots\geq\lambda_{rd}$. 
Assume that either $\{\lambda_{r1}:r\in\mathbb N\}$ is unbounded 
or that $\{\lambda_{rd}:r\in\mathbb N\}$ 
has zero as a cluster point. 
Then, we know by  \cite[Theorem 4.3]{LS2019}
that there exists a subsequence of $(\Sigma_r)_r$, which we again denote by $(\Sigma_r)_r$, 
such that for any fixed $\nu>0$ it holds
\begin{equation}\label{eq:fixed_nu_to_infty}
\lim\limits_{r\to\infty} L_\nu (\Sigma_r)=\infty.
\end{equation}
Since $k\mapsto\left(1+\frac{k}{x}\right)^k$ is monotone increasing, for $\nu_r\geq d+1$ we have 
\begin{align}
(d+\nu_r)\sum_{i=1}^n w_i \log\left(1+\frac1{\nu_r}x_i^\tT \Sigma_r^{-1}x_i\right)
&=\sum_{i=1}^n w_i \log\left(\left(1+\frac1{\nu_r}x_i^\tT \Sigma_r^{-1}x_i\right)^{\nu_r+d}\right)\\
&\geq \sum_{i=1}^n w_i \log\left(\left(1+\frac1{\nu_r}x_i^\tT \Sigma_r^{-1}x_i\right)^{\nu_r}\right)\\
&\geq \sum_{i=1}^n w_i \log\left(\left(1+\frac1{d+1}x_i^\tT \Sigma_r^{-1}x_i\right)^{d+1}\right)\\
&= (d+1)\sum_{i=1}^n w_i \log\left(1+\frac1{d+1}x_i^\tT \Sigma_r^{-1}x_i\right)\\
&\geq (d+1)\sum_{i=1}^n w_i \log\left(1+x_i^\tT \Sigma_r^{-1}x_i\right) - \log(d+1)^{d+1}.
\end{align}
By \eqref{eq:asym_likelihood} this yields
\begin{align}
\lim\limits_{r\to\infty}L(\nu_r,\Sigma_r)
&\geq 
d\log(2) - \log(d+1)^{d+1} + \lim\limits_{r\to\infty} (d+1)\sum_{i=1}^n w_i \log\left(1+x_i^\tT \Sigma_r^{-1}x_i\right)+\log(\abs{\Sigma_r})
\\
&=d\log(2)- \log(d+1)^{d+1} +\lim\limits_{r\to\infty}L_1(\Sigma_r)=\infty.
\end{align}
This contradicts the assumption that $(\nu_r,\Sigma_r)_r$ is a minimizing sequence of $L$.
Hence $\overline{\{\Sigma_r:r\in\mathbb N\}}$ is a bounded subset of $\mathrm{SPD}(d)$.
\\[1ex]
Finally, we show that any subsequence of $(\Sigma_r)_r$ has a subsequence which converges to $\hat\Sigma=\sum_{i=1}^n w_i x_ix_i^\tT$. 
Then the whole sequence $(\Sigma_r)_r$ converges to $\hat \Sigma$.\\
Let $(\Sigma_{r_k})_k$ be a subsequence of $(\Sigma_r)_r$. 
Since it is bounded, it has a convergent subsequence $(\Sigma_{r_{k_l}})_l$ 
which converges to some $\tilde\Sigma\in\overline{\{\Sigma_r:r\in\mathbb N\}}\subset\mathrm{SPD}(d)$. 
For simplicity, we denote $(\Sigma_{r_{k_l}})_l$ again by $(\Sigma_r)_r$. 
Since $(\Sigma_r)_r$ is converges, we know that also $(x_i^\tT \Sigma_r^{-1}x_i)_r$ converges and is bounded. 
By $\lim\limits_{r\to\infty}\nu_r=\infty$ we know that the functions $x\mapsto\left(1+\frac{x}{\nu_r}\right)^{\nu_r}$ 
converge locally uniformly to $x\mapsto \exp(x)$ as $r\to\infty$. 
Thus we obtain
\begin{align}
&
\lim\limits_{r\to\infty}(d+\nu_r)\sum_{i=1}^n w_i \log\left(1+\frac1{\nu_r}x_i^\tT \Sigma_r^{-1}x_i\right)\\
&=
\lim\limits_{r\to\infty}\sum_{i=1}^n w_i\log\left(\left(1+\frac1{\nu_r}x_i^\tT \Sigma_r^{-1}x_i\right)^{d+\nu_r}\right)
\\
&=
\lim\limits_{r\to\infty} \sum_{i=1}^n w_i \log\left(\lim\limits_{r\to\infty}\left(1+\frac1{\nu_r}x_i^\tT 
\Sigma_r^{-1}x_i\right)^{\nu_r}\left(1+\frac1{\nu_r}x_i^\tT \Sigma_r^{-1}x_i\right)^d\right)\\
&=
\lim\limits_{r\to\infty}\sum_{i=1}^n w_i \log\left(\lim\limits_{r\to\infty}\left(1+\frac1{\nu_r}x_i^\tT \Sigma_r^{-1}x_i\right)^{\nu_r}\right)\\
&=
\sum_{i=1}^n w_i \log\left(\exp(x_i^\tT\tilde{\Sigma}^{-1}x_i)\right)=\sum_{i=1}^n w_ix_i^\tT \tilde\Sigma^{-1}x_i.
\end{align}
Hence we have 
\begin{align}
\inf_{(\nu,\Sigma)\in\R_{>0}\times\mathrm{SPD}(d)}L(\nu,\Sigma)=\lim\limits_{r\to\infty} L(\nu_r,\Sigma_r)
=d\log(2)+\sum_{i=1}^n w_ix_i^\tT\tilde\Sigma^{-1}x_i+\log(|\tilde\Sigma|).
\end{align}
By taking the derivative with respect to $\Sigma$ we see that the right-hand side is minimal if and only if 
$\Sigma=\hat\Sigma=\sum_{i=1}^nw_ix_ix_i^\tT$. 
On the other hand, by similar computations as above we get
\begin{align}
\inf_{(\nu,\Sigma)\in\R_{>0}\times\mathrm{SPD}(d)}L(\nu,\Sigma)
&\leq
\lim\limits_{r\to\infty} L(\nu_r,\hat\Sigma)\\
&= d\log(2) + \log(|\hat\Sigma|) 
+ \lim_{v_r \rightarrow \infty} (d+\nu_r) \sum_{i=1}^n w_i \log \big(1+\frac1{\nu_r}x_i^\tT \hat \Sigma^{-1}x_i\big)\\
&= d\log(2) + \log(|\hat\Sigma|) + \sum_{i=1}^n w_ix_i^\tT \hat\Sigma^{-1}x_i+\log(|\hat\Sigma|),
\end{align}
so that $\tilde\Sigma=\hat\Sigma$. This finishes the proof.
\end{proof}

%-------------------------------------------------------------
\section{Zeros of $F$}\label{sec:zero_F}
%-------------------------------------------------------------
%----------------------------------------------------
In this section, we are interested in the existence of solutions of  
\eqref{ML_nu}, i.e., in zeros of $F$ for arbitrary fixed $\mu$ and $\Sigma$. 
Setting $x \coloneqq \frac{\nu}{2} > 0$, $t \coloneqq \frac{d}{2}$ and
$$
s_i \coloneqq \frac12 (x_i - \mu)^\tT \Sigma^{-1} (x_i - \mu), \quad i=1,\ldots,n.
$$
we rewrite the function $F$ in \eqref{ML_nu} as
\begin{align} \label{function_F}
F(x) &= \phi (x) - \phi(x+t) + 
\sum_{i=1}^n w_i \left( \frac{x+t}{x +  s_i}- \log\left(\frac{x + t}{x +  s_i} \right) - 1\right)
\\
&= \sum_{i=1}^n w_i F_{s_i} (x)
=
\sum_{i=1}^n w_i \big( A(x) + B_{s_i}(x) \big),
\end{align}
where 
\begin{equation} \label{Fs}
F_s(x) \coloneqq A(x) + B_s(x)
\end{equation} 
and
\begin{align}\label{A+B}
A(x) \coloneqq \phi (x) - \phi(x+t),\qquad
B_s (x) \coloneqq  \frac{x+t}{x +  s}- \log\left(\frac{x + t}{x +  s} \right) - 1.
\end{align}
The digamma function $\psi$ and $\phi = \psi - \log(\cdot)$ are well examined in the literature, see \cite{AS65}.
The function $\phi(x)$ is the expectation value of a random variable which is $\Gamma(x,x)$ distributed.
It holds $-\frac{1}{x} < \phi(x) < - \frac{1}{2x}$ and it is well-known that 
$-\phi$ is \emph{completely monotone}. This implies that the negative of $A$
is also completely monotone, i.e. for all $x > 0$ and $m \in \mathbb N_0$ we have
\begin{equation}\label{cm_A}
(-1)^{m+1} \phi^{(m)} (x) > 0, \qquad (-1)^{m+1} A^{(m)} (x) > 0,
\end{equation}
in particular $A < 0$, $A' > 0$ and $A'' < 0$.
Further, it is easy to check that
\begin{align}
&\lim_{x\rightarrow 0} \phi(x) = -\infty, \qquad \lim_{x\rightarrow \infty} \phi(x) = 0^-,\label{asymp_phi}\\
&\lim_{x\rightarrow 0} A(x) = -\infty, \qquad \lim_{x\rightarrow \infty} A(x) = 0^-.\label{asymp_A}
\end{align}
On the other hand, we have that $B(x) \equiv 0$ if $s=t$ in which case $F_s = A < 0$ and has therefore no zero.
If $s \not = t$, then $B_s$ is \emph{completely monotone}, i.e., for all $x > 0$ and $m \in \mathbb N_0$,
\begin{equation}\label{cm_B}
 (-1)^m B_s^{(m)} (x) > 0,
\end{equation}
in particular $B_s> 0$, $B_s' < 0$ and $B_s'' >0$,
and
\begin{equation}\label{asymp_B}
B_s(0) = \frac{t}{s} - \log \left( \frac{t}{s} \right) - 1 > 0, \qquad \lim_{x\rightarrow \infty} B_s (x) = 0^+.
\end{equation}
Hence we have
\begin{equation} \label{eq1}
\lim_{x \rightarrow 0} F_s(x) = -\infty, \qquad \lim_{x \rightarrow \infty} F_s(x) = 0.
\end{equation}

If $X \sim {\mathcal N}(\mu,\Sigma)$ is a $d$-dimensional random vector, then $Y \coloneqq (X-\mu)^\tT \Sigma^{-1} (X-\mu) \sim \chi_d^2$
with $\mathbb E (Y) = d$ and $\Var(Y) = 2d$. Thus we would expect 
that for samples $x_i$ from such a random variable $X$ 
the corresponding values $(x_i - \mu)^\tT \Sigma^{-1} (x_i - \mu)$ lie 
with high probability in the interval $[d - \sqrt{2d},d+ \sqrt{2d}]$, respective $s_i \in [t -\sqrt{t}, t + \sqrt{t}]$.
These considerations are reflected in the following theorem and corollary.

\begin{Theorem}\label{thm:ensure}
For $F_s: \R_{>0} \rightarrow \R$ given by \eqref{Fs} the following relations hold true:
\begin{itemize}
\item[i)]
If $s \in [t - \sqrt{t},t+ \sqrt{t}] \cap \mathbb R_{>0}$, then $F_s(x) < 0$ 
for all $x >0$ so that $F_s$ has no zero.
\item[ii)]
If $s > 0$ and $s \not \in [t - \sqrt{t},t+ \sqrt{t}]$,
then there exists 
$x_+$ such that $F_s(x) >0$ for all $x \ge x_+$. 
In particular, $F_s$ has a zero.
\end{itemize}
\end{Theorem}

\begin{proof}
We have 
\begin{align}
F_s'(x) 
&= 
\phi'\left(x\right) - \phi'(x+t) - \frac{(s-t)^2}{(x +s)^2(x+t)}\\
&= 
\psi'(x) -  \psi'(x+t) - \frac{t}{x(x+t)} - \frac{(s-t)^2}{(x +s)^2(x+t)}.
\end{align}
We want to sandwich $F'_s$ between two rational functions $P_s$ and $P_s + Q$
which zeros can easily be described.

Since the trigamma function $\psi'$ has the series representation 
\begin{equation}\label{psi_series}
\psi'(x) = \sum_{k=0}^{\infty} \frac{1}{(x+k)^{2}},
\end{equation}
see~\cite{AS65},
we obtain
\begin{equation} \label{function_1}
F_s'(x) = \sum_{k=0}^\infty\frac{1}{(x+k)^2} - \frac{1}{(x+k+t)^2} - \frac{t}{x(x+t)} - \frac{(s-t)^2}{(x+s)^2(x+t)}.
\end{equation}
For $x > 0$, we have
$$I(x) = \int_0^\infty \underbrace{\frac{1}{(x+u)^2}-\frac{1}{(x+u+t)^2}}_{g(u)} \, du 
=\frac1x-\frac1{x+t} = \frac{t}{(x+t)x}.$$
Let $R(x)$ and $T(x)$ denote the rectangular and trapezoidal rule, respectively, 
for computing the integral with step size 1.
 Then we verify
 \[R(x)=\sum_{k=0}^\infty g(k)=\sum_{k=0}^\infty \frac1{(x+k)^2}-\frac1{(x+k+t)^2}\]
 so that
 \begin{align}
 F_s'(x) &= \left( R(x) - T(x) \right) + \left( T(x) - I(x) \right)  - \frac{(s-t)^2}{(x+s)^2(x+t)}\\
      & = \frac12 \left(\frac{1}{x^2} -\frac{1}{(x+t)^2} \right)+ \left( T(x) - I(x) \right) - \frac{(s-t)^2}{(x+s)^2(x+t)}.
 \end{align}
By considering the first and second derivative of $g$ we see the integrand  in $I(x)$
is strictly decreasing and strictly convex.
Thus,
  $
  P_s(x) < F_s'(x)
  $,
  where
  \begin{align}
  P_s(x) 
  &\coloneqq  \frac12 \left(\frac{1}{x^2} -\frac{1}{(x+t)^2} \right) - \frac{(s-t)^2}{(x+s)^2(x+t)}
  = \frac{(2tx + t^2)(x+s)^2 - (s-t)^2 x^2(x+t)}{2x^2(x+s)^2(x+t)^2}\\
  &= \frac{p_s(x)}{2x^2(x+s)^2(x+t)^2}.
  \end{align}
  with
  $
  p_s(x) \coloneqq a_3 x^3 + a_2 x^2 + a_1 x + a_0
  $
  and
  $$
  a_0 = t^2s^2 > 0,\quad a_1 = 2st(s+t) > 0, \quad a_2 = t(4s+t - (s-t)^2),\quad a_3 = 2\left( t- (s-t)^2 \right) .
  $$
  We have 
  \begin{equation} \label{main_coeff}
	a_3 \ge 0 \quad \Longleftrightarrow \quad s \in [t - \sqrt{t}, t + \sqrt{t}]
	\end{equation}
  and
  $$a_2 \geq 0 \quad \Longleftrightarrow \quad 
	s \in [t+2-\sqrt{4+ 5t}, t+2 + \sqrt{4+ 5t}] \supset [t - \sqrt{t}, t + \sqrt{t}]$$
  for $t \geq 1$. For $t=\frac12$, it holds $[t+2-\sqrt{4+ 5t}, t+2 + \sqrt{4+ 5t}]\supset [0,t+\sqrt{t}]$.
	
	Thus, for $s \in [t - \sqrt{t}, t + \sqrt{t}]$, 
  by the sign rule of Descartes, $p_s(x)$ has no positive zero
	which implies
   $$
  0 \le P_s(x) < F_s'(x) \quad \mathrm{for} \quad s \in [t - \sqrt{t}, t + \sqrt{t}] \cap \mathbb R _{>0}.
  $$
	Hence, the continuous function $F_s$ is monotone increasing  
and by \eqref{eq1} we obtain
$F_s (x) < 0$ for all $x > 0$ if $s \in [t - \sqrt{t}, t + \sqrt{t}] \cap \mathbb R _{>0}$.
\\[1ex]  
Let $s>0$ and $s \not \in [t - \sqrt{t}, t + \sqrt{t}]$.
By
    $$
  T(x)-I(x)=\sum_{k=0}^\infty \left(   \frac12(g(k+1)+g(k)) - \int_0^1 g(k+u) \, du \right)
  $$
and Euler's summation formula, we obtain
  $$
  T(x) - I(x) = \sum_{k=0}^\infty \frac{1}{12} \left( g'(k+1) - g'(k) \right) - \frac{1}{720} g^{(4)}(\xi_k), \quad \xi_k \in (k,k+1)
  $$
  with $g'(u) = -\frac{2}{(x+u)^3}+\frac{2}{(x+u+t)^3}$ and $g^{(4)}(u) =  \frac{5!}{(x+u)^6}-\frac{5!}{(x+u+t)^6}$,
  so that
  \begin{align} \label{**}
  T(x) - I(x) =&  -\frac{1}{12} g'(0) +\sum_{k=0}^\infty \frac16\frac1{(x+\xi_k+t)^6}-\frac16\frac1{(x+\xi_k)^6}\\
  <&- \frac{1}{12}g'(0)
	=\frac16\frac{3t x^2  + 3t^2x + t^3}{x^3(x+t)^3}.
   \end{align}
  Therefore, we conclude
  $$
  F_s'(x) < P_s(x) + \underbrace{\frac16\frac{3t x^2  + 3t^2x + t^3}{x^3(x+t)^3}}_{Q(x)} = 
	\frac{p_s(x) x (x+t) + (t x^2  + t^2x + \frac13 t^3)(x+s)^2}{2 x^3(x+s)^2(x+t)^3} 
	$$
  The main coefficient of $x^5$ of the polynomial in the numerator is $2(t-(s-t)^2)$ which fulfills \eqref{main_coeff}.
	Therefore, if $s \not \in [t - \sqrt{t}, t + \sqrt{t}]$, then there exists $x_+$ large enough 
	such that the numerator becomes smaller than zero for all $x \ge x_+$.
 Consequently, $F'_s(x) \leq P_s(x) + Q(x)<0$ for all $x \geq x_+$.
  Thus, $F_s$ is decreasing on $[x_+,\infty)$. 
	By \eqref{eq1}, we conclude that $F_s$ has a zero.
	\end{proof}

The following corollary states that $F_s$ has exactly one zero if $s > t+ \sqrt{t}$.	
Unfortunately we do not have such a results for $s < t - \sqrt{t}$.
	
\begin{Corollary}
Let $F_s: \R_{>0} \rightarrow \R$ be given by \eqref{Fs}.	If $s >t + \sqrt{t}$, $t \ge 1$, then 
$F_s$ has exactly one zero.
\end{Corollary}

\begin{proof}
	By Theorem \ref{thm:ensure}ii) 	and since $\lim_{x\rightarrow 0} F_s(x) = -\infty$ and 
	$\lim_{x\rightarrow \infty} = 0^+$, it remains to prove that $F_s'$ has at most one zero.
	Let $x_0>0$ be the smallest number such that $F_s'(x_0)=0$.
	We prove that $F_s'(x)<0$  for all $x>x_0$.
	To this end, we show that $h_s(x)\coloneqq F_s'(x)(x+s)^2(x+t)$ is strictly decreasing. 
	By \eqref{function_1} we have
	\begin{align} \label{function_h}
	h_s(x) &= (x+s)^2(x+t)\left(\sum_{k=0}^\infty\frac{1}{(x+k)^2} - \frac{1}{(x+k+t)^2} - \frac{t}{x(x+t)} \right)- (s-t)^2,
	\end{align}
	and for  $s>t$ further 
	\begin{align} 
	h_s'(x) 
	&= \left(2(x+s)(x+t)+ (x+s)^2\right)\left(\sum_{k=0}^\infty\frac{1}{(x+k)^2} - \frac{1}{(x+k+t)^2} - \frac{t}{x(x+t)} \right) \\
	&\quad + (x+s)^2(x+t)\left(\sum_{k=0}^\infty\frac{-2}{(x+k)^3} + \frac{2}{(x+k+t)^3} + \frac{t(2x+t)}{x^2(x+t)^2} \right)\\
	&\leq 3(x+s)^2 \left(\sum_{k=0}^\infty\frac{1}{(x+k)^2} - \frac{1}{(x+k+t)^2} - \frac{t}{x(x+t)} \right)\\
	&\quad + (x+s)^2(x+t)\left(\sum_{k=0}^\infty\frac{-2}{(x+k)^3} + \frac{2}{(x+k+t)^3} + \frac{t(2x+t)}{x^2(x+t)^2} \right). \\
	&= (x+s)^2 (R(x)-I(x)),
	\end{align}
	where $I(x)$ is the integral and $R(x)$ the corresponding rectangular rule 
	with step size 1 of the function  $g\coloneqq g_1 + g_2$ defined as
	\begin{equation}
	g_1(u)\coloneqq 3\left(  \frac{1}{(x+u)^2}  - \frac{1}{(x+ t + u)^2}\right),
	\quad 
	g_2(u)\coloneqq(x+t)\left( \frac{-2}{(x+u)^3} + \frac{2}{(x+t+ u)^3}\right).
	\end{equation}
	We show that $R(x)-I(x) <0$ for all $x>0$. 
	Let $T(x)$, $T_i(x)$ be the trapezoidal rules with step size 1 corresponding to $I(x)$ and 
	$I_i(x)=\int_{0}^{\infty} g_i(u)du$, $i=1,2$. 
	Then it follows
	$$
	R(x)- I(x) = R(x) - T(x) + T(x) - I(x) 
	=R(x) - T(x) + T_1(x) - I_1(x)  + T_2(x) - I_2(x). 
	$$
	Since $g_2$ is a decreasing, concave function,  we conclude $T_2(x) - I_2(x)<0$. 
	Using Euler's summation formula in \eqref{**} for $g_1$, we get
	\begin{align}	
	T_1(x) - I_1(x) &= -\frac{1}{12}g_1'(0) - \frac{1}{720}\sum_{k=0}^{\infty} g_1^{(4)}(\xi_k), \quad \xi_k\in(k,k+1).
	\end{align} 
	Since $g_1^{(4)}$ is a positive function, we can write
	\begin{align}
	R(x) - I(x) &< R(x) - T(x) + T_1(x) - I_1(x) \leq \frac{1}{2} g(0)  -\frac{1}{12}g_1'(0)\\
	&= \frac{3}{2}\left( \frac{1}{x^2}-\frac{1}{(x+t)^2}\right) +
	\frac{1}{2}(x+t) \left( \frac{-2}{x^3} + \frac{2}{(x+t)^3}\right)  -
	\frac{1}{2}\left( \frac{-1}{x^3} + \frac{1}{(x+t)^3}\right)\\
	&=\frac{t}{2} \, \frac{(- 3 t + 3 )x^2  +(- 5 t^2 + 3t)x -2 t^3  +t^2}{x^3(x+t)^3}.
	\end{align}
	All coefficients of $x$ are smaller or equal than zero for $t \ge 1$ which implies that $h_s$
	is strictly decreasing.
\end{proof}	

Theorem \ref{thm:ensure} implies the following corollary.
	
\begin{Corollary}\label{cor:ensure}
For $F: \R_{>0} \rightarrow \R$ given by \eqref{function_F} and 
$\delta_i \coloneqq (x_i - \mu)^\tT \Sigma^{-1} (x_i - \mu)$, $i=1,\ldots,n$, the following relations hold true:
\begin{itemize}
\item[i)]
If $\delta_i \in [d - \sqrt{2d},d+ \sqrt{2d}] \cap \mathbb R_{>0}$ for all $i\in \{1,\ldots,n\}$, then $F(x) < 0$ 
for all $x >0$ so that $F$ has no zero.
\item[ii)]
If $\delta_i > 0$ and $\delta_i \not \in [d - \sqrt{2d},d+ \sqrt{2d}]$ for all $i\in \{1,\ldots,n\}$,
there exists 
$x_+$ such that $F(x) >0$ for all $x \ge x_+$. 
In particular, $F$ has a zero.
\end{itemize}
\end{Corollary}

\begin{proof}
 Consider $F = \sum_{i=1}^n F_{s_i}$. 
	If $\delta_i \in [d - \sqrt{2d},d+ \sqrt{2d}] \cap \mathbb R_{>0}$ for all $i\in \{1,\ldots,n\}$,
	then we have by Theorem \ref{thm:ensure} that $F_{s_i} (x) < 0$ for all $x>0$.
	Clearly, the same holds true for the whole function $F$ such that it cannot have a zero.
	
	If $\delta_i \not \in [d - \sqrt{2d},d+ \sqrt{2d}]$ for all $i\in \{1,\ldots,n\}$, then we know
	by Theorem \ref{thm:ensure} that there exist $x_{i+} > 0$ such that $F_{s_i} (x) > 0$ for 
	$x \ge x_{i+}$. Thus, $F(x) > 0$ for $x \ge x_+ \coloneqq \max_i(x_{i+})$. 
	Since $\lim_{x \rightarrow 0} F(x) = -\infty$
	this implies that $F$ has a zero.	
\end{proof}

%--------------------------------------------------------------------------------------------------------------------
\section{Algorithms} \label{sec:algs}
%--------------------------------------------------------------------------------------------------------------------
In this section, we propose an alternative of the classical EM algorithm 
for computing the parameters of the Student-$t$ distribution along with convergence results.
In particular, we are interested in estimating the degree of freedom parameter $\nu$,
where the function $F$ is of particular interest. 
\\

\textbf{Algorithm \ref{alg:EM}} with weights $w_i = \frac{1}{n}$, $i=1,\ldots,n$, 
is the classical EM algorithm.
Note that the function in the third M-Step
\begin{align} \label{poly_EM}
\Phi_r \left( \frac{\nu}{2} \right)  
&\coloneqq
\phi \left( \frac{\nu}{2} \right) 
\underbrace{ - \, \phi \left( \frac{\nu_r + d}{2} \right)		          
						 + \sum_{i=1}^n w_i \left( \gamma_{i,r} - \log( \gamma_{i,r} ) - 1 \right)}_{c_r}
\end{align}
has a unique zero since by \eqref{asymp_phi} the function $\phi < 0$ 
is monotone increasing with $\lim_{x \rightarrow \infty} \phi(x) = 0^-$ and $c_r > 0$.
Concerning the convergence of the EM algorithm it is known 
that the  values of the objective function $L(\nu_r,\mu_r,\Sigma_r)$ are monotone decreasing in $r$ and that
a subsequence of the iterates
converges to a critical point of $L(\nu,\mu,\Sigma)$ if such a point exists, see \cite{Byrne2017}.
\\

%-----------------------------------------
\begin{algorithm}[!ht]
	\caption{EM Algorithm (EM)} \label{alg:EM}
	\begin{algorithmic}
		\State \textbf{Input:} $x_1,\ldots,x_n\in \R^d$, $n \geq d+1$, $w \in \mathring \Delta_n$ 
		\State \textbf{Initialization:} 
		$\nu_0 = \eps>0$,  $\mu_0 =\frac{1}{n} \sum\limits_{i=1}^n x_i$, 
		$\Sigma_0 =\frac{1}{n}\sum\limits_{i=1}^n (x_i-\mu_0)(x_i-\mu_0)^\tT$
		\For{$r=0,\ldots$}
		\vspace{0.2cm}
		
		 \textbf{E-Step:} Compute the weights
		\begin{align*} 	
		\delta_{i,r} &=  (x_i-\mu_r)^\tT \Sigma_r^{-1} (x_i-\mu_r)\\
		\gamma_{i,r} &=    \frac{\nu_r + d}{ \nu_r + \delta_{i,r} }
		\end{align*}
		
		\hspace*{0.2cm} \textbf{M-Step:} Update the parameters
		\begin{align*}
		\mu_{r+1}    
		&=   
		\frac{ \sum\limits_{i=1}^{n} w_i \gamma_{i,r} x_i}{ \sum\limits_{i=1}^{n} w_i\gamma_{i,r} } 
		\\
		\Sigma_{r+1} 
		&=   
		\sum\limits_{i=1}^{n} w_i \gamma_{i,r} (x_i-\mu_{r+1})(x_i-\mu_{r+1})^\tT 
		\\
		\nu_{r+1}& \; = \;
		\text{ zero of } \;
		 \phi\left(\frac{\nu}{2}\right) -\phi\left( \frac{\nu_r + d}{2}\right)
		+ 
				\sum_{i=1}^n w_i\left( \gamma_{i,r} - \log(\gamma_{i,r} ) - 1 \right)
		\end{align*}
		\EndFor
	\end{algorithmic}
\end{algorithm}

\textbf{Algorithm \ref{alg:aEM}} distinguishes from the EM algorithm in the iteration of $\Sigma$, where the
factor $\frac{1}{\sum\limits_{i=1}^n w_i \gamma_{i,r}}$ is incorporated now.
The computation of this factor requires no additional computational effort,
but speeds up the performance in particular for smaller $\nu$.
Such kind of acceleration was suggested in \cite{KTV94,MVD97}.
\emph{For fixed $\nu \ge 1$}, it was shown in \cite{vanDyk1995} that this algorithm is indeed an EM algorithm
arising from another choice of the hidden variable than used in the standard approach, see also \cite{Laus2019}. 
Thus, it follows for fixed $\nu \ge 1$ that the sequence $L(\nu ,\mu_r,\Sigma_r)$ is monotone decreasing.
However, we also iterate over $\nu$. In contrast to the EM Algorithm \ref{alg:EM} 
our $\nu$ iteration step depends on $\mu_{r+1}$ and $\Sigma_{r+1}$
instead of $\mu_{r}$ and $\Sigma_{r}$. This is important for our convergence results.
Note that for both cases, the accelerated algorithm can  no longer be interpreted as an EM algorithm, so that
the convergence results of the classical EM approach are no longer available.

Let us mention that a Jacobi variant of Algorithm \ref{alg:aEM} for \emph{fixed} $\nu$ i.e.
$$
		\Sigma_{r+1} 
		=   
		\sum\limits_{i=1}^{n} \frac{w_i\gamma_{i,r} (x_i-\mu_{r})(x_i-\mu_{r})^\tT }{\sum_{i=1}^n w_i \gamma_{i,r}},
		$$
		with $\mu_r$ instead of $\mu_{r+1}$ including a convergence proof was suggested in \cite{LS2019}.
		The main reason for this index choice was that we were able to prove monotone convergence of a simplified version
of the algorithm for estimating the location and scale of Cauchy noise ($d=1$, $\nu = 1$) 
which could be not achieved with the variant incorporating $\mu_{r+1}$, see \cite{LPS18}.
This simplified version is known as myriad filter in image processing. 
In this paper, we keep the original variant from the EM algorithm \eqref{eq:aem} since we are mainly interested in the
computation of $\nu$.

Instead of the above algorithms we suggest to take the critical point equation \eqref{ML_nu} more
directly into account in the next two algorithms.
\\

%-----------------------------------------
\begin{algorithm}[!ht]
	\caption{Accelerated EM-like Algorithm (aEM)} \label{alg:aEM}
	\begin{algorithmic}
	\State	Same as Algorithm \ref{alg:EM} except for 
%	 \begin{equation} \label{eq:aem}
%		\Sigma_{r+1} 
%		=   
%		\sum\limits_{i=1}^{n} \frac{w_i\gamma_{i,r} (x_i-\mu_{r+1})(x_i-\mu_{r+1})^\tT }{\sum_{i=1}^n w_i \gamma_{i,r}}
%		\end{equation}
	 \begin{align} 
		\Sigma_{r+1} 
		&=   
		\sum\limits_{i=1}^{n} \frac{w_i\gamma_{i,r} (x_i-\mu_{r+1})(x_i-\mu_{r+1})^\tT }{\sum_{i=1}^n w_i \gamma_{i,r}}\label{eq:aem}\\
        \nu_{r+1}& \; = \;
		\text{ zero of } \;
		 \phi\left(\frac{\nu}{2}\right) -\phi\left( \frac{\nu_r + d}{2}\right)
		+ 
				\sum_{i=1}^n w_i\left( \frac{\nu_r+d}{\nu_r+\delta_{i,r+1}} - \log\left(\frac{\nu_r+d}{\nu_r+\delta_{i,r+1}} \right) - 1 \right)
		\end{align}
	\end{algorithmic}
\end{algorithm}
%----------------------------------------------------------------------------------------------------------------------

\textbf{Algorithm \ref{alg:MMF}} computes a zero of 
\begin{equation} \label{eq:alternative}
\Psi_r \left( \frac{\nu}{2} \right)  
\coloneqq
%\phi \left( \frac{\nu}{2} \right) - \phi \left( \frac{\nu + d}{2} \right)
%		+ \underbrace{ \sum_{i=1}^n w_i\left( \gamma_{i,r} - \log(\gamma_{i,r} ) - 1 \right) }_{b_r}
\phi \left( \frac{\nu}{2} \right) - \phi \left( \frac{\nu + d}{2} \right)
		+ \underbrace{ \sum_{i=1}^n w_i\left( \frac{\nu_r+d}{\nu_r+\delta_{i,r+1}} - \log\left(\frac{\nu_r+d}{\nu_r+\delta_{i,r+1}} \right) - 1 \right) }_{b_r}
\end{equation}
This function has a unique zero since by \eqref{asymp_A} the function $A(x) = \phi(x) -\phi(x+t) < 0$ is monotone increasing with 
$\lim_{x \rightarrow \infty} A(x)= 0_-$ and $b_r > 0$.
\\

%-----------------------------------------
\begin{algorithm}[!ht]
	\caption{Multivariate Myriad Filter (MMF)} \label{alg:MMF}
	\begin{algorithmic}
	\State	Same as Algorithm \ref{alg:aEM} except for
	\begin{equation} \label{eq:one_step}
        \nu_{r+1} \; = \;
		\text{ zero of } \;
		 \phi\left(\frac{\nu}{2}\right) -\phi\left( \frac{\nu + d}{2}\right)
		+ 
				\sum_{i=1}^n w_i\left( \frac{\nu_r+d}{\nu_r+\delta_{i,r+1}} - \log\left(\frac{\nu_r+d}{\nu_r+\delta_{i,r+1}} \right) - 1 \right)
%	\nu_{r+1} \; = \; 
%		\text{ zero of } \;
%	 \phi \left( \frac{\nu}{2} \right) - \phi \left( \frac{\nu + d}{2} \right)
%		+  \sum_{i=1}^n w_i\left( \gamma_{i,r} - \log(\gamma_{i,r} ) - 1 \right) 
\end{equation}
	\end{algorithmic}
\end{algorithm}
%----------------------------------------------------------------------------------------------------------------------

Finally, \textbf{Algorithm \ref{alg:GMMF}} computes the update of $\nu$ 
by directly finding a zero of the whole function $F$ in \eqref{ML_nu} given $\mu_r$ and $\Sigma_r$.
The existence of such a zero was discussed in the previous section.
The zero computation is done by an inner loop which iterates the update step of $\nu$ from Algorithm \ref{alg:MMF}.
We will see that the iteration converge indeed to a zero of $F$.
\\

%----------------------------------------------------------------------------------------------------------------------
\begin{algorithm}[!ht]
	\caption{General Multivariate Myriad Filter (GMMF)} \label{alg:GMMF}
	\begin{algorithmic}
		\State	Same as Algorithm \ref{alg:aEM} except for 
		\begin{align}
		\nu_{r+1}&=  \; \text{ zero of } 
		 \phi\left( \frac{\nu}{2} \right) 
			-\phi\left( \frac{\nu +d}{2} \right)
			+ \sum_{i=1}^n w_i \left( \frac{\nu + d}{\nu +  \delta_{i,r+1}}  - \log\left( \frac{\nu + d}{\nu +  \delta_{i,r+1}} \right) - 1 \right)	 
			\end{align}
		  \For{$l=0,\ldots$}
		\begin{align}
		&\nu_{r,0} = \nu_r\\
				&\nu_{r,l+1}  \; \text{ zero of } \; \phi\left( \frac{\nu}{2} \right) 
			-\phi\left( \frac{\nu +d}{2} \right)
			+ \sum_{i=1}^n w_i \left( \frac{\nu_{r,l} + d}{\nu_{r,l} +  \delta_{i,r+1}}  - 
			\log\left( \frac{\nu_{r,l} + d}{\nu_{r,l} +  \delta_{i,r+1}} \right) - 1 \right)	 
		\end{align}
		  \EndFor
	\end{algorithmic}
\end{algorithm}
%-------------------------------------------------------------------------------------------

In the rest of this section, we prove that the sequence $(L(\nu_r,\mu,r,\Sigma_r))_r$ generated by Algorithm \ref{alg:aEM} and \ref{alg:MMF} decreases in each iteration step
and that there exists  a subsequence of the iterates which 
converges to a critical point.

We will need the following auxiliary lemma.

\begin{Lemma}\label{prop:inner_loop}
Let $F_a,F_b\colon \R_{>0}\to\R$ be continuous functions, 
where $F_a$ is strictly increasing and $F_b$ is strictly decreasing. 
Define $F\coloneqq F_a + F_b$. 
For any initial value $x_0 >0$ assume that the sequence generated by
$$
x_{l+1} = \text{ zero of } F_a(x)+F_b(x_l)
$$
is uniquely determined, i.e., the functions on the right-hand side have a unique zero.
Then it holds
\begin{enumerate}
\item[i)] If $F(x_0)<0$, then $(x_l)_l$ is strictly increasing and $F(x) < 0$ for all $x \in [x_l,x_{l+1}]$, $l \in \mathbb N_0$.
\item[ii)] If $F(x_0)>0$, then $(x_l)_l$ is strictly decreasing and $F(x) >0$ for all $x \in [x_{l+1},x_{l}]$, $l \in \mathbb N_0$.
\end{enumerate}
Furthormore, assume that there exists 
$x_->0$ with $F(x) <0$ for all $x<x_-$ 
and 
$x_+>0$ with $F(x)>0$ for all $x>x_+$. 
Then, the sequence $(x_l)_l$ converges to a zero $x^*$ of $F$.
\end{Lemma}

\begin{proof}
We consider the case i) that $F(x_0)<0$. Case ii) follows in a similar way.

We show by induction that $F(x_l)<0$ and that $x_{l+1} > x_l$ for all $l \in \mathbb N$.
Then it holds for all $l\in\mathbb N$ and $x \in(x_l,x_{l+1})$ that 
$F_a(x) + F_b(x) < F_a(x) + F_b(x_l) < F_a(x_{l+1} ) + F_b(x_l) = 0$. 
Thus $F(x) < 0$ for all $x \in [x_l,x_{l+1}]$, $l \in \mathbb N_0$.
\\[1ex]
\textbf{Induction step.} Let $F_a(x_l)+F_b(x_l)<0$. 
Since $F_a(x_{l+1})+F_b(x_l) =0 > F_a(x_l)+F_b(x_l)$ 
and $F_a$ is strictly increasing, 
we have $x_{l+1}>x_l$. 
Using that $F_b$ is strictly decreasing, we get  $F_b(x_{l+1})<F_b(x_l)$
and consequently
$$
F(x_{l+1}) = F_a(x_{l+1}) + F_b(x_{l+1}) < F_a(x_{l+1}) + F_b(x_l)=0.
$$

Assume now that $F(x)>0$ for all $x>x_+$.
Since the sequence $(x_l)_l$ is strictly increasing and $F(x_l) < 0$
it must be bounded from above by $x_+$. 
Therefore it converges to some $x^*\in \R_{>0}$. 
Now, it holds by the continuity of $F_a$ and $F_b$ that
$$
0 =\lim\limits_{l\to\infty} F_a(x_{l+1}) + F_b(x_l) = F_a(x^*) + F_b(x^*) = F(x^*).
$$
Hence $x^*$ is a zero of $F$.
\end{proof}

For the setting in Algorithm \ref{alg:GMMF}, Lemma \ref{prop:inner_loop} implies the following corollary.

\begin{Corollary}
Let $F_a (\nu) \coloneqq  \phi \left( \frac{\nu}{2} \right) - \phi\left(\frac{\nu+d}{2} \right)$ 
and 

$F_b (\nu) \coloneqq \sum_{i=1}^n w_i \left( \frac{\nu + d}{\nu +  \delta_{i,r+1}}  - 
			\log\left( \frac{\nu + d}{\nu +  \delta_{i,r+1}} \right) - 1 \right)$, $r \in \N_0$
Assume that there exists $\nu_+ > 0$ such that $F \coloneqq F_a + F_b >0$ for all $\nu \ge \nu_+$.
Then, the sequence $(\nu_{r,l})_l$ generated by the $r$-th inner loop of Algorithm \ref{alg:GMMF}
converges to a zero of $F$.
\end{Corollary}

Note that by Corollary \ref{cor:ensure} the above condition on $F$ is fulfilled in each iteration step, e.g. 
if $\delta_{i,r} \not \in [d - \sqrt{2d} , d + \sqrt{2d}]$ for $i=1,\ldots,n$ and $r \in\N_0$.

\begin{proof}
From the previous section we know that $F_a$ is strictly increasing and $F_b$ is strictly decreasing. 
Both functions are continuous.
If $F(\nu_r) < 0$, then we know from Lemma \ref{prop:inner_loop} that $(\nu_{r,l})_l$ is increasing and converges to a zero $\nu_r^*$
of $F$.

If $F(\nu_r) > 0$, then we know from Lemma \ref{prop:inner_loop} that $(\nu_{r,l})_l$ is decreasing.
The condition that there exists 
$x_-\in\R_{>0}$ with $F(x) <0$ for all $x<x_-$ is fulfilled since $\lim_{x \rightarrow 0} F(x) = -\infty$.
Hence, by Lemma \ref{prop:inner_loop}, the sequence converges to a zero $\nu_r^*$
of $F$.
\end{proof}

To prove that the objective function decreases in each step of the Algorithms \ref{alg:aEM} - \ref{alg:GMMF} we need the following lemma.

\begin{Lemma}\label{lem:function_decreasing}
Let $F_a,F_b\colon \R_{>0}\to\R$ be continuous functions, 
where $F_a$ is strictly increasing and $F_b$ is strictly decreasing. 
Define $F\coloneqq F_a + F_b$ and let $G\colon \R_{>0}\to \R$ be an antiderivative of $F$, i.e. 
$F= \frac{\mathrm{d}}{\mathrm{d}x} G$. 
For an arbitrary $x_0 >0$, let $(x_{l})_l$ be the sequence generated by
$$
x_{l+1} = \text{ zero of } F_a(x) + F_b(x_l).
$$
Then the following holds true:
\begin{enumerate}
\item[i)] The sequence $(G(x_{l}))_l$ is monotone decreasing with $G(x_l)=G(x_{l+1})$ 
if and only if $x_0$ is a critical point of $G$. 
If $(x_{l})_l$ converges, then the limit $x^*$ fulfills 
$$
G(x_0) \geq G(x_1) \geq G(x^*),
$$
with equality if and only if $x_0$ is a critical point of $G$. 
\item[ii)] Let $F = \tilde F_a + \tilde F_b$ be another splitting of $F$ 
with continuous functions $\tilde F_a, \tilde F_b$, where the first one is strictly increasing and the second one strictly decreasing.
Assume that $\tilde F_a'(x) > F_a'(x)$ for all $x>0$. 
Then holds for $y_1 \coloneqq \text{ zero of } \tilde F_a(x) + \tilde F_b(x_0)$ 
that $G(x_0) \geq G(y_1) \geq G(x_1)$ 
with equality if and only if $x_0$ is a critical point of $G$.
\end{enumerate}
\end{Lemma}

\begin{proof}
i) If $F(x_0)=0$, then $x_0$ is a critical point of $G$.

Let $F(x_0)<0$. By Lemma \ref{prop:inner_loop} we know that $(x_{l})_l$ is strictly increasing
and that $F(x) < 0$ for $x \in [x_r,x_{r+1}]$, $r \in \N_0$.
By the Fundamental Theorem of calculus it holds
$$
G(x_{l+1})=G(x_{l})+\int_{x_{l}}^{x_{l+1}} F(\nu) d\nu.
$$
Thus, $G(x_{l+1})<G(x_{l})$. 

Let $F(x_0)>0$. By Lemma \ref{prop:inner_loop} we know that $(x_{l})_l$ is strictly decreasing
and that $F(x) > 0$ for $x \in [x_{r+1},x_{r}]$, $r \in \N_0$.
Then $$
G(x_{l}) = G(x_{l+1})+\int_{x_{l+1}}^{x_{l}} F(\nu) d\nu.
$$
implies $G(x_{l+1})<G(x_{l})$.
Now, the rest of assertion i) follows immediately.
\\[1ex]
ii) It remains to show that $G(x_1)\leq G(y_1)$. Let $F(x_0) <0$. 
Then we have $y_1\geq x_0$ and   $x_1\geq x_0$. 
By the Fundamental Theorem of calculus we obtain
\begin{align}
F(x_0) + \int_{x_0}^{x_1} F_a'(x)dx      &= F_a(x_0)+\int_{x_0}^{x_1} F_a'(x) dx + F_b (x_0) = F_a (x_1) + F_b (x_0)=0,\\
F(x_0) + \int_{x_0}^{y_1} \tilde F_a'(x)dx&=\tilde F_a(x_0)+\int_{x_0}^{y_1}\tilde F_a'(x) dx+\tilde F_b(x_0) =\tilde F_a(y_1)+\tilde F_b(x_0)=0.
\end{align}
This yields
\begin{equation}\label{eq:int_equal}
\int_{x_0}^{x_1} F_a'(x) dx=\int_{x_0}^{y_1}\tilde F_a'(x)dx,
\end{equation}
and since $\tilde F'_a(x) > F'_a(x)$ further  $y_1\leq x_1$ with equality if and only if $x_0=x_1$, 
i.e., if $x_0$ is a critical point of $G$.
Since $F(x)<0$ on $(x_0,x_1)$ it holds
$$
G(x_1)=G(y_1)+\int_{y_1}^{x_1}F(x) dx \leq G(y_1),
$$
with equality if and only if $x_0=x_1$.
The case $F(x_0) >0$ can be handled similarly.
\end{proof}

Lemma \ref{lem:function_decreasing} implies the following relation between the values of the objective function 
$L$ for Algorithms \ref{alg:aEM} - \ref{alg:GMMF}.

\begin{Corollary}\label{lem:likelihood_decreasing_nu}
For the same fixed $\nu_r>0, \mu_r\in\R^d, \Sigma_r\in\mathrm{SPD}(d)$ define $\mu_{r+1}$, $\Sigma_{r+1}$,
$\nu_{r+1}^{\mathrm{aEM}}$, $\nu_{r+1}^{\mathrm{MMF}}$ and $\nu_{r+1}^{\mathrm{GMMF}}$ by Algorithm \ref{alg:aEM}, \ref{alg:MMF} and \ref{alg:GMMF},
respectively. For the GMMF algorithm assume that the inner loop converges. 
Then it holds
$$
L(\nu_r,\mu_{r+1},\Sigma_{r+1}) 
\geq L(\nu_{r+1}^{\mathrm{aEM}},\mu_{r+1},\Sigma_{r+1})
\geq L(\nu_{r+1}^{\mathrm{MMF}},\mu_{r+1},\Sigma_{r+1})
\geq L(\nu_{r+1}^{\mathrm{GMMF}},\mu_{r+1},\Sigma_{r+1}).
$$
Equality holds true if and only if $\frac{\mathrm{d}}{\mathrm{d}\nu}L(\nu_r,\mu_{r+1},\Sigma_{r+1})=0$ and
in this case $\nu_{r} = \nu_{r+1}^{\mathrm{aEM}} = \nu_{r+1}^{\mathrm{MMF}} = \nu_{r+1}^{\mathrm{GMMF}}$.
\end{Corollary}

\begin{proof}
For $G(\nu) \coloneqq L(\nu,\mu_{r+1},\Sigma_{r+1})$, we have 
$\frac{\mathrm{d}}{\mathrm{d}\nu} L(\nu,\mu_{r+1},\Sigma_{r+1}) = F(\nu)$,
where
$$
F(\nu) \coloneqq \phi\left( \frac{\nu}{2} \right) 
			-\phi\left( \frac{\nu +d}{2} \right)
			+ \sum_{i=1}^n w_i \left( \frac{\nu + d}{\nu +  \delta_{i,r+1}}  - 
			\log\left( \frac{\nu + d}{\nu +  \delta_{i,r+1}} \right) - 1 \right).
$$
We use the splitting 
$$F = F_a + F_b = \tilde F_a + \tilde F_b$$
with 
$$
F_a (\nu)\coloneqq 
\phi\left(\frac\nu2 \right)- \phi\left(\frac{\nu + d}{2} \right), \quad
\tilde F_a \coloneqq \phi\left(\frac\nu2 \right)
$$ 
and 
$$
F_b(\nu) \coloneqq 
\sum_{i=1}^n w_i \left( \frac{\nu + d}{\nu +  \delta_{i,r+1}}  - 
			\log\left( \frac{\nu + d}{\nu +  \delta_{i,r+1}} \right) - 1 \right), \quad
\tilde F_b (\nu)\coloneqq - \phi \left(\frac{\nu+d}{2} \right) + F_b(\nu).
$$
By the considerations in the previous section we know that $F_a$, $\tilde F_a$ are strictly increasing
and $F_b$, $\tilde F_b$ are strictly decreasing.
Moreover, since $\phi' > 0$ we have $\tilde F'_a > F'_a$.
Hence it follows from Lemma \ref{lem:function_decreasing}(ii) that 
$L(\nu_r,\mu_{r+1},\Sigma_{r+1}) \ge L(\nu_r^{\mathrm{aEM}},\mu_{r+1},\Sigma_{r+1}) \ge  L(\nu_r^{\mathrm{MMF}},\mu_{r+1},\Sigma_{r+1})$.
Finally, we conclude by Lemma \ref{lem:function_decreasing}(i) that 
$L(\nu_r^{\mathrm{MMF}},\mu_{r+1},\Sigma_{r+1}) \ge L(\nu_r^{\mathrm{GMMF}},\mu_{r+1},\Sigma_{r+1})$.
\end{proof}

Concerning the convergence of the three algorithms we have the following result.
%{\color{blue}\\Das Folgende geht nur, falls wir im $\nu$-Schritt bereits $\mu_{r+1}$ und $\Sigma_{r+1}$ verwenden.}

\begin{Theorem}\label{cor:likelihood_decreases}
Let $(\nu_r,\mu_r,\Sigma_r)_r$ be sequence generated by Algorithm \ref{alg:aEM}, \ref{alg:MMF} or \ref{alg:GMMF}, respectively
starting with arbitrary initial values $\nu_0 >0,\mu_0\in\R^d,\Sigma_0\in\mathrm{SPD}(d)$.
For the GMMF algorithm we assume that in each step the inner loop converges. 
Then it holds for all $r\in\mathbb N_0$ that
$$
L(\nu_r,\mu_r,\Sigma_r) \geq L(\nu_{r+1},\mu_{r+1},\Sigma_{r+1}),
$$
with equality if and only if $(\nu_r,\mu_r,\Sigma_r)=(\nu_{r+1},\mu_{r+1},\Sigma_{r+1})$.
\end{Theorem}

\begin{proof}
By the general convergence results of the accelerated EM algorithm for fixed $\nu$, see also
\cite{LS2019}, it holds
$$
L(\nu_r,\mu_{r+1},\Sigma_{r+1})\leq L(\nu_r,\mu_r,\Sigma_r),
$$
with equality if and only if $(\mu_r,\Sigma_r)=(\mu_{r+1},\Sigma_{r+1})$. 
By Corollary \ref{lem:likelihood_decreasing_nu} it holds
$$
L(\nu_{r+1},\mu_{r+1},\Sigma_{r+1})\leq L(\nu_r,\mu_{r+1},\Sigma_{r+1}),
$$
with equality if and only if $\nu_r=\nu_{r+1}$. 
The combination of both results proves the claim.
\end{proof}

\begin{Lemma}\label{lem:Tcont}
Let $T = (T_1, T_2, T_3): \R_{>0} \times \R^d \times \SPD(d) \rightarrow \R_{>0} \times \R^d \times \SPD(d)$ 
be the operator of one iteration step of Algorithm \ref{alg:aEM} (or \ref{alg:MMF}). 
Then $T$ is continuous.
\end{Lemma}

\begin{proof}
We show the statement for Algorithm \ref{alg:MMF}. For Algorithm \ref{alg:aEM} it can be shown analogously. 
Clearly the mapping $(T_2,T_3) (\nu,\mu,\Sigma)$ is continuous. 
Since 
$$T_1(\nu,\mu,\Sigma) = \text{zero of } \Psi(x, \nu,T_2(\nu,\mu,\Sigma),T_3(\nu,\mu,\Sigma)),$$ 
where
\begin{align}
\Psi(x,\nu,\mu,\Sigma)
&=\phi\left(\frac{x}{2}\right)-\phi\left(\frac{x+d}{2}\right)\\
&+\sum_{i=1}^n w_i\left(\frac{\nu+d}{\nu+(x_i-\mu)^T\Sigma^{-1}(x_i-\mu)}-\log\left(\frac{\nu+d}{\nu+(x_i-\mu)^T\Sigma^{-1}(x_i-\mu)}\right)-1\right).
\end{align}
It is sufficient to show that the zero of $\Psi$ depends continuously on $\nu$, $T_2$ and $T_3$. 
Now the continuously differentiable function $\Psi$ is strictly increasing in $x$, so that  $\frac{\partial}{\partial x} \Psi(x,\nu,T_2,T_3)>0$. 
By $\Psi(T_1,\nu,T_2,T_3)=0$, the Implicit Function Theorem yields the following statement: 
%-------------------
There exists an open neighborhood $U\times V$ of $(T_1,\nu,T_2,T_3)$ with $U\subset\R_{>0}$ 
and $V\subset \R_{>0}\times\R^d\times\SPD(d)$ and a continuously differentiable function $G\colon V\to U$ 
such that for all $(x,\nu,\mu,\Sigma)\in U\times V$ it holds
$$
\Psi(x,\nu,\mu,\Sigma)=0 \quad \text{if and only if}\quad G(\nu,\mu,\Sigma)=x.
$$
Thus the zero of $\Psi$ depends continuously on $\nu$, $T_2$ and $T_3$.
\end{proof}

This implies the following theorem.

\begin{Theorem}
Let $(\nu_r,\mu_r,\Sigma_r)_r$ be the sequence generated by Algorithm \ref{alg:aEM} or \ref{alg:MMF} 
with arbitrary initial values $\nu_0 >0,\mu_0\in\R^d,\Sigma_0\in\mathrm{SPD}(d)$. 
Then every cluster point of $(\nu_r,\mu_r,\Sigma_r)_r$ is a critical point of $L$.
\end{Theorem}

\begin{proof}
The mapping $T$ defined in Lemma \ref{lem:Tcont} is continuous. Further we know from its definition that $(\nu,\mu,\Sigma)$ 
is a critical point of $L$ if and only if it is a fixed point of $T$. Let $(\hat\nu,\hat\mu,\hat\Sigma)$ be a cluster point of $(\nu_r,\mu_r,\Sigma_r)_r$. 
Then there exists a subsequence $(\nu_{r_s},\mu_{r_s},\Sigma_{r_s})_s$ which converges to $(\hat\nu,\hat\mu,\hat\Sigma)$. 
Further we know by Theorem \ref{cor:likelihood_decreases} that $L_r=L(\nu_r,\mu_r,\Sigma_r)$ is decreasing. 
Since $(L_r)_r$ is bounded from below, it converges. Now it holds
\begin{align}
L(\hat \nu,\hat \mu,\hat \Sigma)&=\lim_{s\to\infty}L(\nu_{r_s},\mu_{r_s},\Sigma_{r_s})\\
&=\lim_{s\to\infty}L_{r_s}=\lim_{s\to\infty}L_{r_s+1}\\
&=\lim_{s\to\infty}L(\nu_{r_s+1},\mu_{r_s+1},\Sigma_{r_s+1})\\
&=\lim_{s\to\infty}L(T(\nu_{r_s},\mu_{r_s},\Sigma_{r_s}))=L(T(\hat\nu,\hat\mu,\hat\Sigma)).
\end{align}
By Theorem \ref{cor:likelihood_decreases} and the definition of $T$ we have that $L(\nu,\mu,\Sigma)=L(T(\nu,\mu,\Sigma))$ if and only if $(\nu,\mu,\Sigma)=T(\nu,\mu,\Sigma)$. By the definition of the algorithm this is the case if and only if $(\nu,\mu,\Sigma)$ is a critical point of $L$. Thus $(\hat\nu,\hat\mu,\hat\Sigma)$ is a critical point of $L$.
\end{proof}

%------------------------------------------------------------------------------------
\section{Numerical Results} \label{sec:numerics}
%------------------------------------------------------------------------------------
In this section we give two numerical examples of the developed theory. First,  we compare   the four different algorithms in Subsection \ref{sec:comp}.
Then, in Subsection \ref{sec:accel}, we address further accelerations of our algorithms by
SQUAREM \cite{VR2008} and DAAREM \cite{HV2019} and show also a comparison with the ECME algorithm \cite{LR95}.
Finally, in Subsection \ref{sec:images}, we provide an application in image analysis by determining the degree of freedom parameter
in images corrupted by Student-$t$ noise.

%------------------------------------------------------------------------------------
\subsection{Comparison of Algorithms} \label{sec:comp}
%------------------------------------------------------------------------------------
In this section, we compare the numerical performance of the classical EM algorithm \ref{alg:EM}
and the proposed Algorithms \ref{alg:aEM}, \ref{alg:MMF} and \ref{alg:GMMF}. To this aim, we did the following Monte Carlo simulation: Based on the stochastic representation of the Student-$t$ distribution, see equation~\eqref{stochastic_representation},   we draw $n=1000$ i.i.d. realizations of the $T_\nu(\mu,\Sigma)$  distribution
with location parameter $\mu=0$ and different scatter matrices $\Sigma$ 
and degrees of freedom parameters $\nu$. Then, we used Algorithms \ref{alg:aEM}, \ref{alg:MMF} and \ref{alg:GMMF} to compute the ML-estimator $(\hat\nu,\hat\mu,\hat\Sigma)$. 
 
We initialize all algorithms with  the sample mean for $\mu$ and the sample covariance matrix for $\Sigma$. Furthermore, we set $\nu=3$ and in all algorithms the zero of the respective function is computed by Newtons Method.
As a stopping criterion we use the following relative distance:
$$
\frac{ \sqrt{ \| \mu_{r+1} - \mu_r \|^2 + \| \Sigma_{r+1} -\Sigma_r \|_F^2} }{ \sqrt{\|\mu_r\|^2+\|\Sigma_r\|_F^2} } + \frac{ \sqrt{(\log(\nu_{r+1})-\log(\nu_r))^2}}{\abs{\log(\nu_r)}}<10^{-5}.
$$ 
We take the logarithm of   $\nu$ in the stopping criterion, because $T_\nu(\mu,\Sigma)$ converges to the normal distribution as $\nu\to\infty$ 
and therefore the difference between $T_\nu(\mu,\Sigma)$ and $T_{\nu+1}(\mu,\Sigma)$ becomes small for large $\nu$.

To quantify the performance  of the algorithms, we count the number of iterations until the stopping criterion is reached. 
Since the inner loop of the GMMF is potentially time consuming we additionally measure the execution time until the stopping criterion is reached. 
This experiment is repeated $N=10.000$ times for different values of $\nu\in\{1,2,5,10\}$.
Afterward we calculate the average number of iterations and the average execution times. 
The results are given in Table \ref{tab:performance}.
We observe that the performance of the algorithms depends on $\Sigma$. 
Further we see, that the performance of the aEM algorithm 
is always better than those of the classical EM algorithm. 
Further all algorithms need a longer time to estimate large $\nu$. 
This seems to be natural since the likelihood function becomes very flat for large $\nu$. 
Further, the GMMF needs the lowest number of iterations. 
But for small $\nu$ the execution time of the GMMF 
is larger than those of the MMF and the aEM algorithm. 
This can be explained by the fact, 
that the $\nu$ step has a smaller relevance for small $\nu$ but is still time consuming in the GMMF. 
The MMF needs slightly more iterations 
than the GMMF but if $\nu$ is not extremely large the execution time is smaller 
than for the GMMF and for the aEM algorithm. In summary, the MMF algorithm is proposed as algorithm of choice.

\begin{table}[htp]
\begin{center}
\resizebox*{!}{8cm}{
\begin{tabular}{c|c|c c c c}
$\Sigma$ & $\nu$ & EM & aEM & MMF & GMMF\\\hline
\multirow{5}{7em}{$\left(\begin{array}{cc}0.1&0\\0&0.1\end{array}\right)$}
         & $1$&$\n62.32\pm\n2.50$&$\n23.44\pm\n0.79$&$22.16\pm0.75$&$\mathbf{20.61\pm0.70}$\\
         & $2$&$\n46.17\pm\n1.82$&$\n26.42\pm\n1.08$&$21.48\pm0.94$&$\mathbf{17.79\pm0.80}$\\
         & $5$&$\n50.42\pm11.22$&$\n49.97\pm\n7.48$&$25.28\pm2.61$&$\mathbf{12.14\pm1.73}$\\
         & $10$&$122.62\pm31.74$&$117.40\pm31.65$&$38.16\pm4.51$&$\mathbf{14.32\pm0.96}$\\
         & $100$&$531.07\pm91.41$&$528.14\pm92.19$&$53.66\pm6.98$&$\mathbf{10.76\pm2.07}$\\\hline
\multirow{5}{7em}{$\left(\begin{array}{cc}1&0\\0&1\end{array}\right)$}
         & $1$&$\n62.34\pm\n2.52$&$\n23.43\pm\n0.78$&$22.16\pm0.75$&$\mathbf{20.59\pm0.70}$\\
         & $2$&$\n46.20\pm\n1.81$&$\n26.43\pm\n1.07$&$21.49\pm0.94$&$\mathbf{17.79\pm0.80}$\\
         & $5$&$\n50.68\pm10.86$&$\n50.06\pm\n7.42$&$25.31\pm2.58$&$\mathbf{12.06\pm1.75}$\\
         & $10$&$122.72\pm31.65$&$117.51\pm31.56$&$38.18\pm4.50$&$\mathbf{14.28\pm0.97}$\\
         & $100$&$531.75\pm90.98$&$528.84\pm91.75$&$53.62\pm6.94$&$\mathbf{10.64\pm2.02}$\\\hline
\multirow{5}{7em}{$\left(\begin{array}{cc}10&0\\0&10\end{array}\right)$}
         & $1$&$\n62.35\pm\n2.55$&$\n23.44\pm\n0.78$&$22.15\pm0.76$&$\mathbf{20.59\pm0.71}$\\
         & $2$&$\n46.27\pm\n1.82$&$\n26.45\pm\n1.08$&$21.51\pm0.95$&$\mathbf{17.81\pm0.80}$\\
         & $5$&$\n50.71\pm11.21$&$\n50.15\pm\n7.61$&$25.34\pm2.63$&$\mathbf{12.08\pm1.78}$\\
         & $10$&$122.44\pm30.66$&$117.19\pm30.56$&$38.17\pm4.46$&$\mathbf{14.27\pm0.96}$\\
         & $100$&$533.21\pm89.80$&$530.27\pm90.57$&$53.64\pm6.93$&$\mathbf{10.62\pm2.01}$\\\hline
\multirow{5}{7em}{$\left(\begin{array}{cc}2&-1\\-1&2\end{array}\right)$}
         & $1$&$\n62.32\pm\n2.55$&$\n23.43\pm\n0.78$&$22.15\pm0.76$&$\mathbf{20.60\pm0.70}$\\
         & $2$&$\n46.22\pm\n1.82$&$\n26.43\pm\n1.09$&$21.50\pm0.94$&$\mathbf{17.80\pm0.80}$\\
         & $5$&$\n50.76\pm11.12$&$\n50.21\pm\n7.52$&$25.35\pm2.59$&$\mathbf{12.09\pm1.75}$\\
         & $10$&$122.37\pm31.01$&$117.17\pm30.92$&$38.13\pm4.49$&$\mathbf{14.30\pm0.96}$\\
         & $100$&$530.89\pm91.36$&$527.96\pm92.15$&$53.68\pm7.07$&$\mathbf{10.75\pm2.08}$
\end{tabular}\,
}%\hfill
\vspace{0.5cm}
\resizebox*{!}{8cm}{
\begin{tabular}{c|c|c c c c}
$\Sigma$ & $\nu$ & EM & aEM & MMF & GMMF\\\hline
\multirow{5}{7em}{$\left(\begin{array}{cc}0.1&0\\0&0.1\end{array}\right)$}
         & $1$&$0.008469\pm0.00111$&$0.003511\pm0.00044$&$\mathbf{0.003498\pm0.00044}$&$0.006954\pm0.00114$\\
         & $2$&$0.006428\pm0.00069$&$0.003995\pm0.00042$&$\mathbf{0.003409\pm0.00036}$&$0.005388\pm0.00061$\\
         & $5$&$0.007237\pm0.00208$&$0.007768\pm0.00181$&$0.004133\pm0.00085$&$\mathbf{0.003752\pm0.00100}$\\
         & $10$&$0.017421\pm0.00532$&$0.017991\pm0.00567$&$0.006187\pm0.00122$&$\mathbf{0.005796\pm0.00110}$\\
         & $100$&$0.070024\pm0.01306$&$0.075191\pm0.01418$&$0.008146\pm0.00131$&$\mathbf{0.005601\pm0.00097}$\\\hline
\multirow{5}{7em}{$\left(\begin{array}{cc}1&0\\0&1\end{array}\right)$}
         & $1$&$0.008645\pm0.00090$&$0.003581\pm0.00034$&$\mathbf{0.003572\pm0.00036}$&$0.007126\pm0.00098$\\
         & $2$&$0.006431\pm0.00074$&$0.003989\pm0.00044$&$\mathbf{0.003417\pm0.00039}$&$0.005427\pm0.00071$\\
         & $5$&$0.006883\pm0.00162$&$0.007352\pm0.00128$&$0.003939\pm0.00058$&$\mathbf{0.003550\pm0.00079}$\\
         & $10$&$0.016434\pm0.00439$&$0.016964\pm0.00470$&$0.005869\pm0.00089$&$\mathbf{0.005493\pm0.00077}$\\
         & $100$&$0.072309\pm0.01507$&$0.077724\pm0.01624$&$0.008363\pm0.00155$&$\mathbf{0.005773\pm0.00117}$\\\hline
\multirow{5}{7em}{$\left(\begin{array}{cc}10&0\\0&10\end{array}\right)$}
         & $1$&$0.008839\pm0.00108$&$0.003664\pm0.00043$&$\mathbf{0.003639\pm0.00042}$&$0.007217\pm0.00104$\\
         & $2$&$0.006516\pm0.00075$&$0.004054\pm0.00048$&$\mathbf{0.003449\pm0.00039}$&$0.005428\pm0.00065$\\
         & $5$&$0.007293\pm0.00207$&$0.007799\pm0.00180$&$0.004149\pm0.00082$&$\mathbf{0.003740\pm0.00098}$\\
         & $10$&$0.020598\pm0.00659$&$0.021193\pm0.00683$&$0.007228\pm0.00167$&$\mathbf{0.006834\pm0.00155}$\\
         & $100$&$0.078682\pm0.01969$&$0.084275\pm0.02087$&$0.009039\pm0.00213$&$\mathbf{0.006246\pm0.00160}$\\\hline
\multirow{5}{7em}{$\left(\begin{array}{cc}2&-1\\-1&2\end{array}\right)$}
         & $1$&$0.008837\pm0.00107$&$0.003648\pm0.00039$&$\mathbf{0.003641\pm0.00041}$&$0.007207\pm0.00104$\\
         & $2$&$0.006481\pm0.00070$&$0.004016\pm0.00041$&$\mathbf{0.003433\pm0.00036}$&$0.005413\pm0.00061$\\
         & $5$&$0.006968\pm0.00167$&$0.007440\pm0.00129$&$0.003965\pm0.00055$&$\mathbf{0.003561\pm0.00077}$\\
         & $10$&$0.016608\pm0.00442$&$0.017107\pm0.00468$&$0.005920\pm0.00092$&$\mathbf{0.005499\pm0.00076}$\\
         & $100$&$0.072354\pm0.01509$&$0.077586\pm0.01619$&$0.008385\pm0.00153$&$\mathbf{0.005715\pm0.00114}$
\end{tabular}
}
\end{center}
\caption{Average number of iterations (top) and execution times (bottom) and the corresponding standard deviations of the different algorithms.}
\label{tab:performance}
\end{table}

In Figure \ref{fig:conv_speed_plots} we exemplarily show the functional values $L(\nu_r,\mu_r,\Sigma_r)$ 
of the four algorithms and samples generated for different values of $\nu$ and $\Sigma=I$.
Note that the $x$-axis of the plots is in log-scale. 
We see that the convergence speed (in terms of number of iterations) 
of the EM algorithm is much slower than those of the MMF/GMMF.  
For small $\nu$ the convergence speed 
of the aEM algorithm is close to the GMMF/MMF, 
but for large $\nu$ it is close to the EM algorithm.

In Figure \ref{fig:hist_plots} we show the histograms of the 
$\nu$-output of $1000$ runs for different values of $\nu$ and $\Sigma=I$. 
Since the $\nu$-outputs of all algorithms are very close together we only plot the output of the GMMF. We see that the accuracy of the estimation of $\nu$ decreases for increasing $\nu$. This can be explained by the fact, that the likelihood function becomes very flat for large $\nu$ such that the estimation of $\nu$ becomes much harder.
%Only for $\nu=100$ the $\nu$-outputs of the GMMF and MMF differ from the outputs 
%of the aEM algorithm. Here,  we give the histograms for both cases. 
%Since variation of the values is quite high we plot the histogram of the logarithms of the $\nu$-outputs of the algorithms.
%We see that the $\nu_r$ of the GMMF and MMF are greater in the case that a minimum of $L$ does not exist. %{\color{red} Macht das Sinn? Ggfs. Achsenbeschriftung anpassen}

%----------------------------------------------------------------------------

\begin{figure}
\centering
\begin{subfigure}[t]{0.5\textwidth}
\centering
\includegraphics[width=0.95\textwidth]{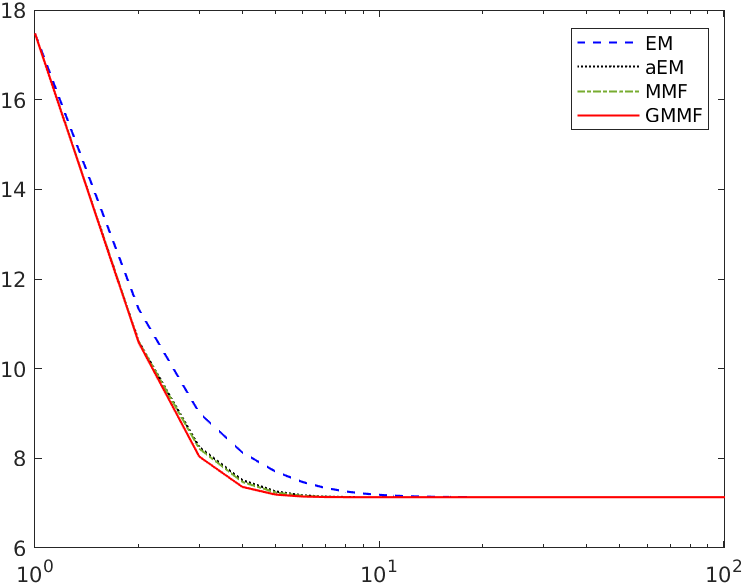}
\caption{$\nu=1$.}
\end{subfigure}\hfill
\begin{subfigure}[t]{0.5\textwidth}
\centering
\includegraphics[width=0.95\textwidth]{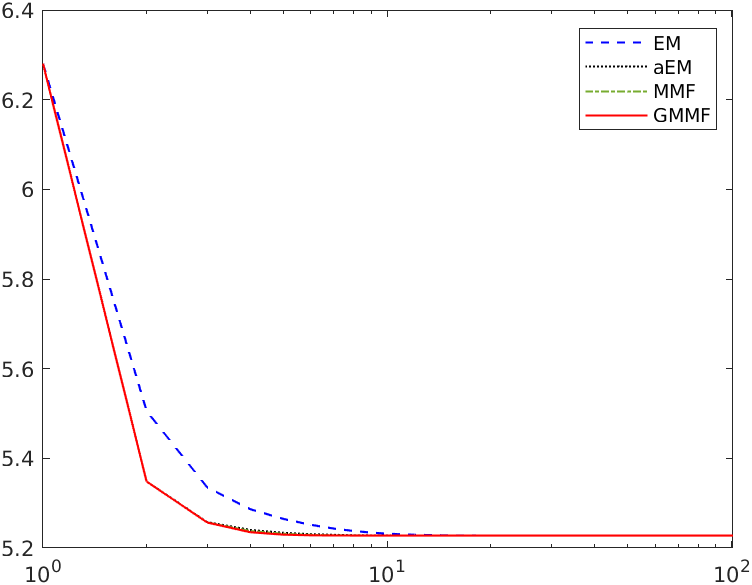}
\caption{$\nu=2$.}
\end{subfigure}\hfill
\begin{subfigure}[t]{0.5\textwidth}
\centering
\includegraphics[width=0.95\textwidth]{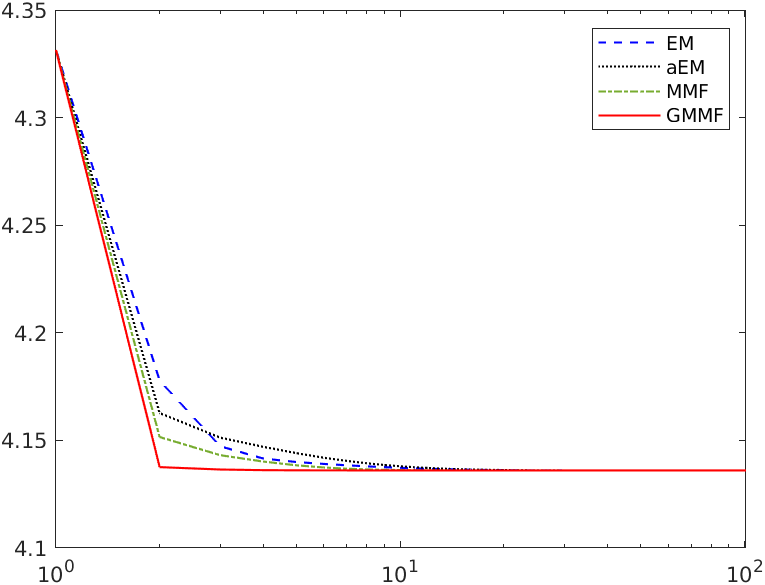}
\caption{$\nu=5$.}
\end{subfigure}\hfill
\begin{subfigure}[t]{0.5\textwidth}
\centering
\includegraphics[width=0.95\textwidth]{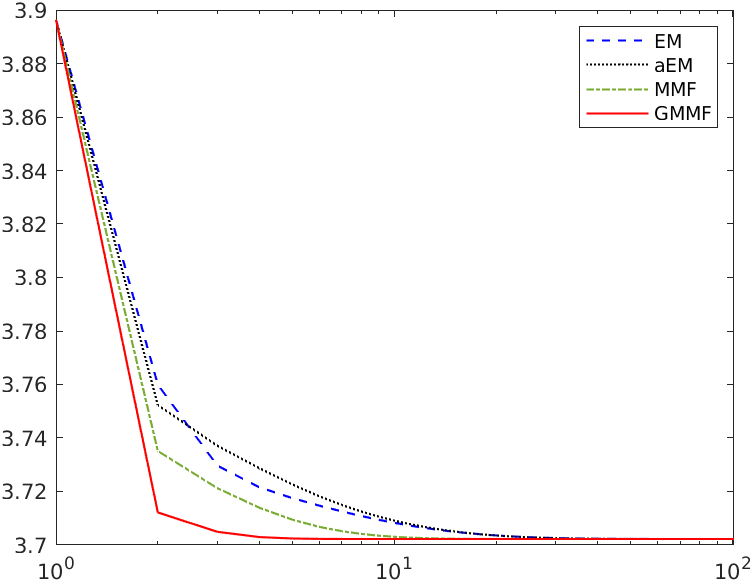}
\caption{$\nu=10$.}
\end{subfigure}
\begin{subfigure}[t]{0.5\textwidth}
\centering
\includegraphics[width=0.95\textwidth]{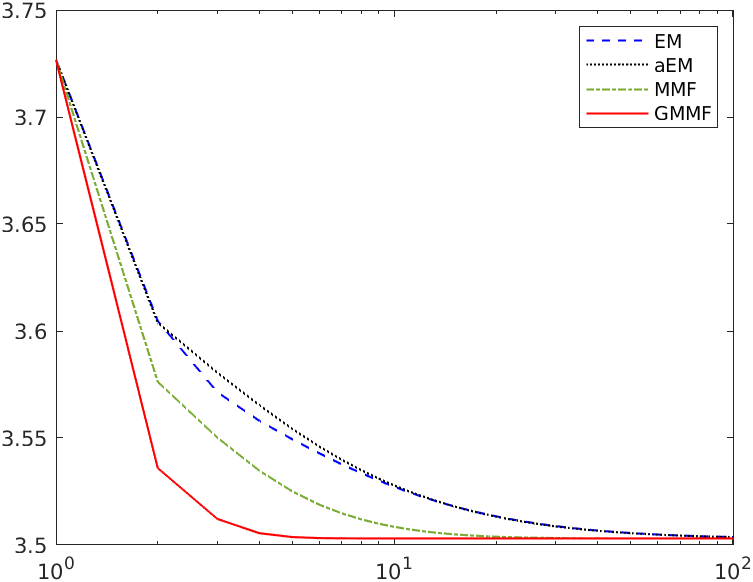}
\caption{$\nu=100$.}
\end{subfigure}\hfill
\begin{subfigure}[t]{0.5\textwidth}
\centering
\includegraphics[width=0.95\textwidth]{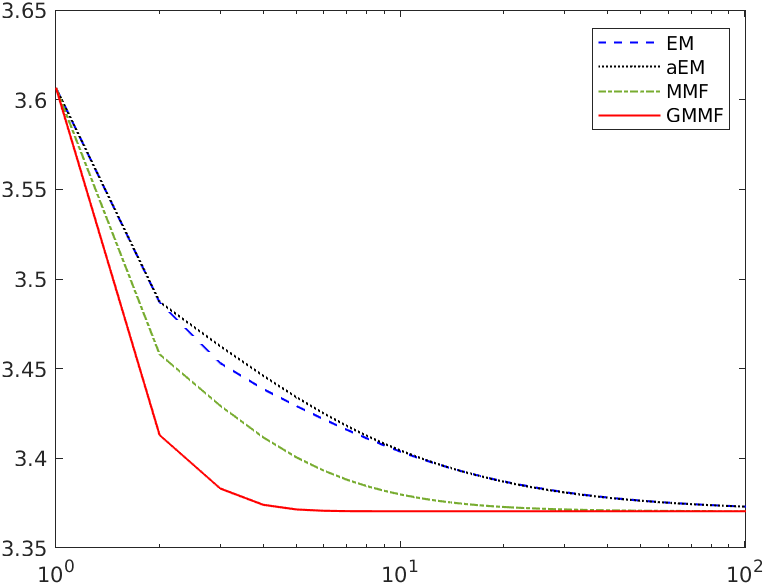}
\caption{$\nu=200$.}
\end{subfigure}\hfill
\caption{Plots of $L(\nu_r,\mu_r,\Sigma_r)$ on the y-axis and $r$ on the x-axis for all algorithms. 
}
\label{fig:conv_speed_plots}
\end{figure}

\begin{figure}
\centering
\begin{subfigure}[t]{0.5\textwidth}
\centering
\includegraphics[width=0.95\textwidth]{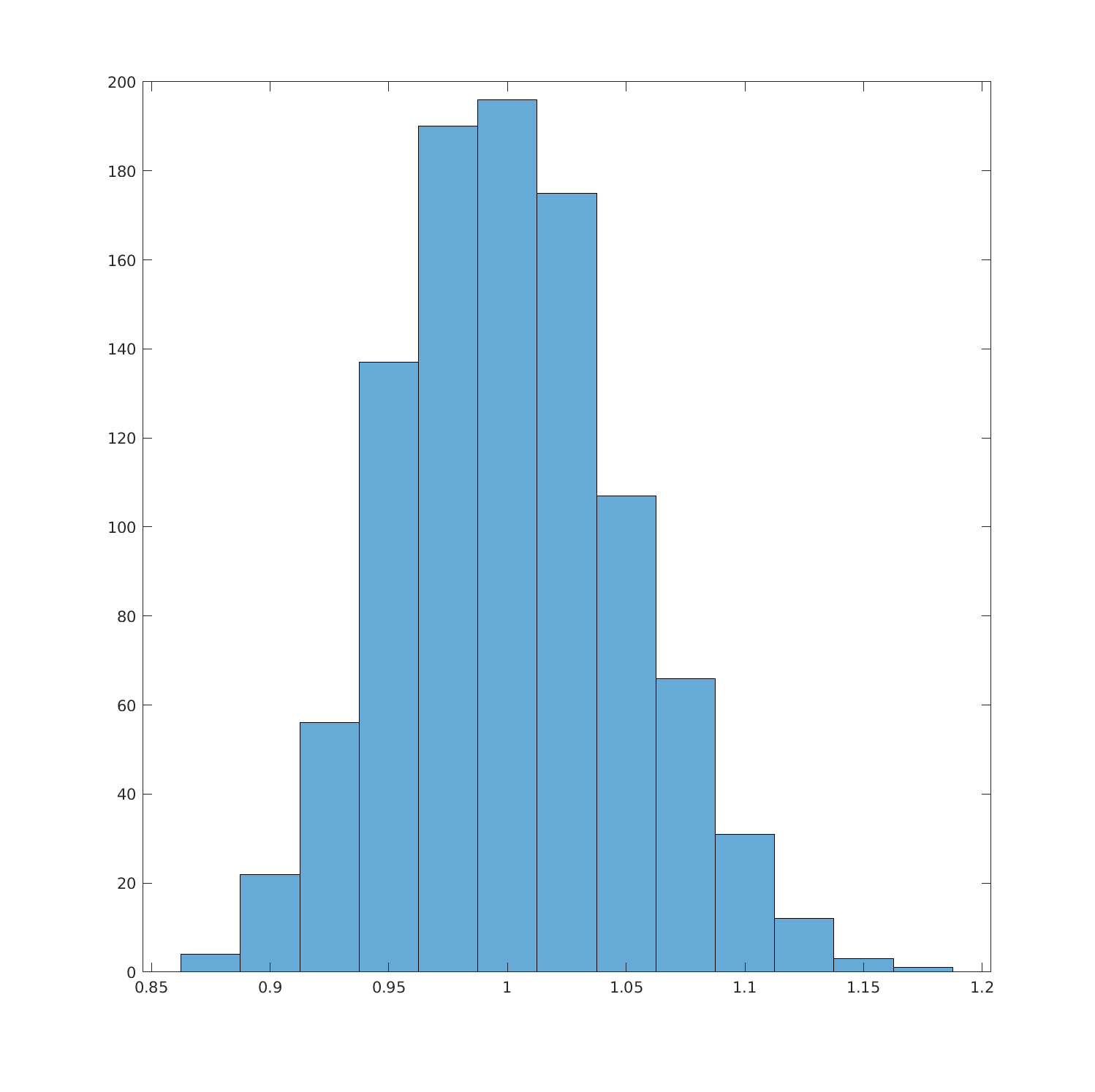}
\caption{$\nu=1$.}
\end{subfigure}\hfill
\begin{subfigure}[t]{0.5\textwidth}
\centering
\includegraphics[width=0.95\textwidth]{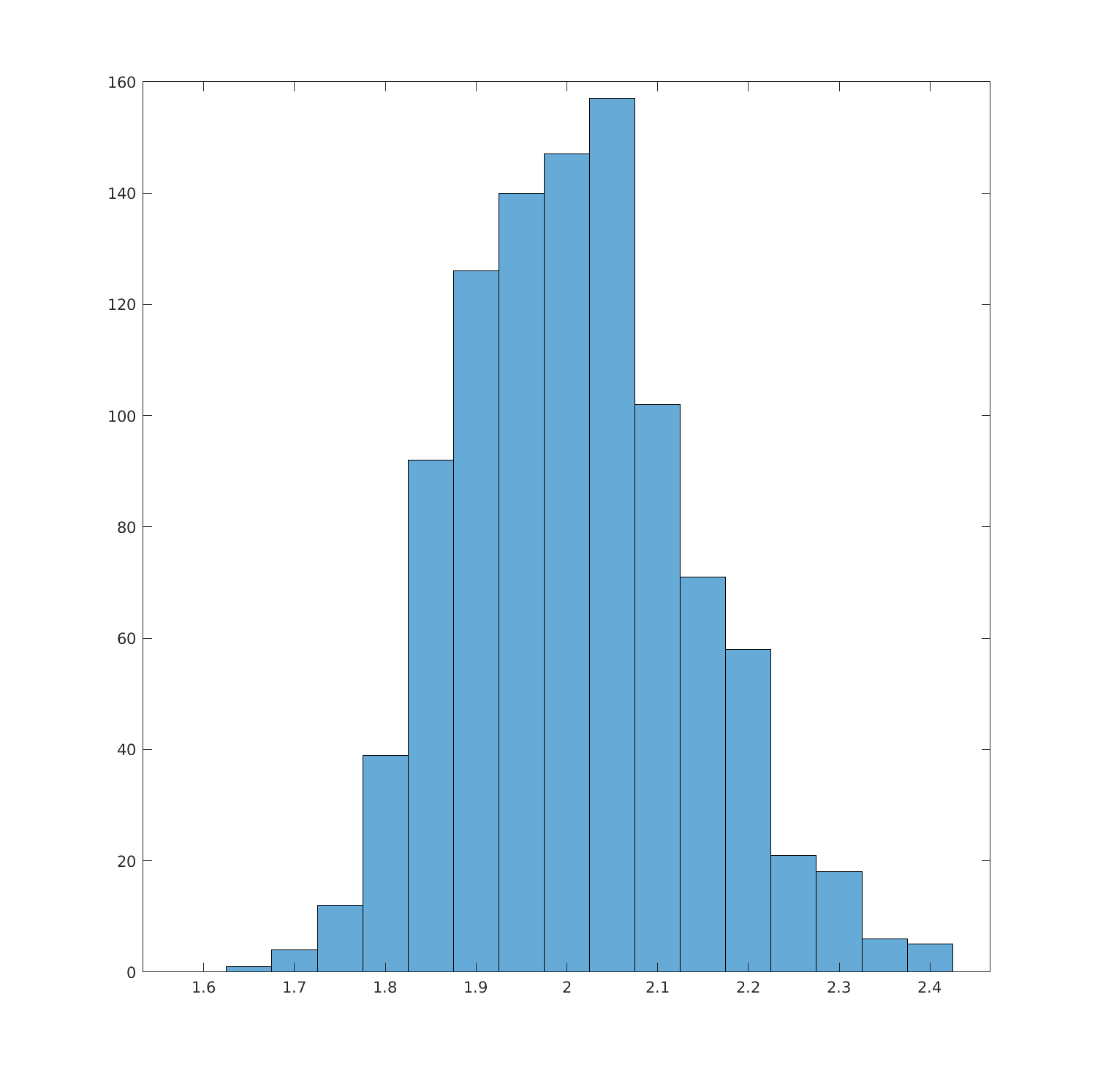}
\caption{$\nu=2$.}
\end{subfigure}\hfill
\begin{subfigure}[t]{0.5\textwidth}
\centering
\includegraphics[width=0.95\textwidth]{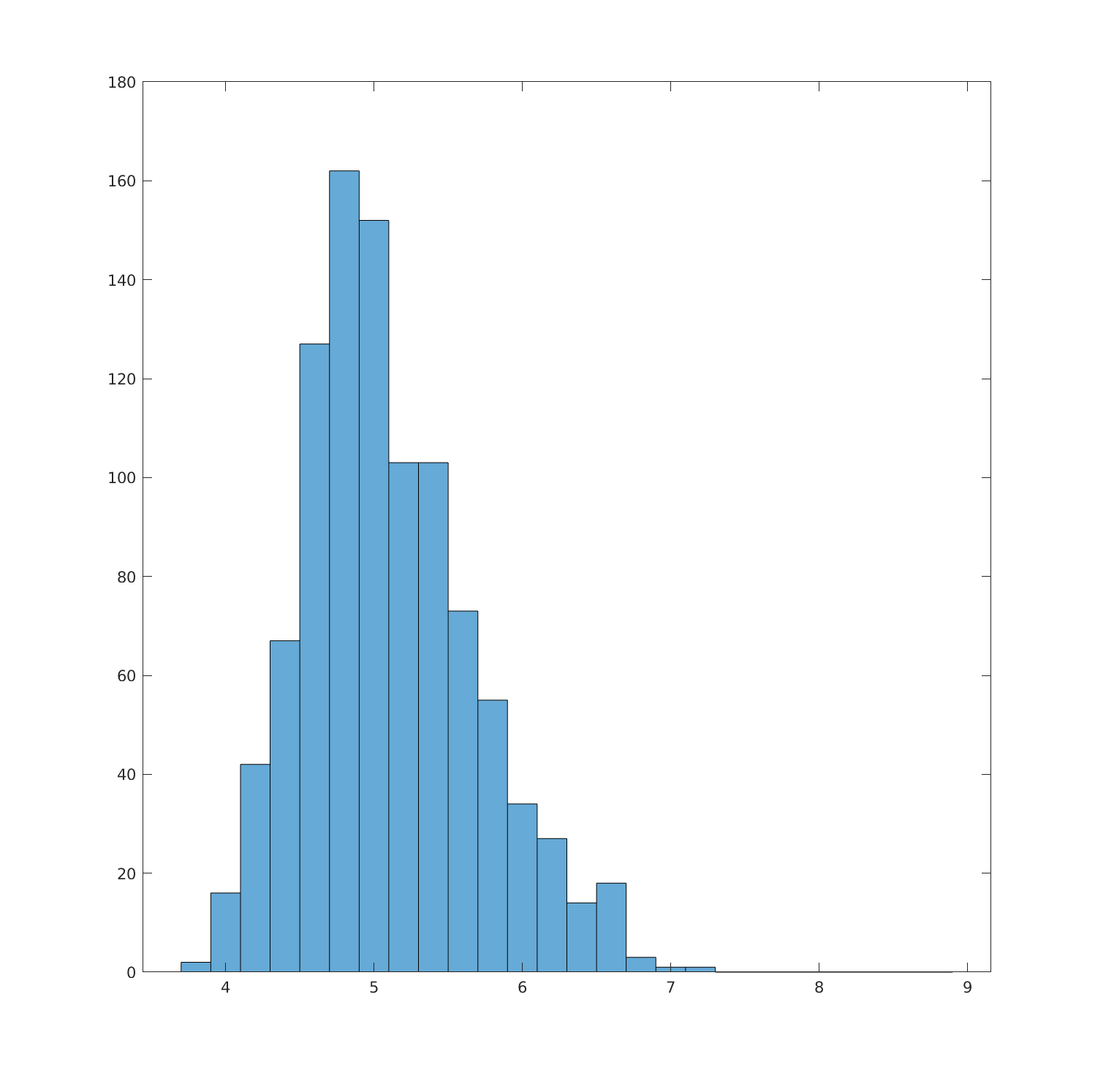}
\caption{$\nu=5$.}
\end{subfigure}\hfill
\begin{subfigure}[t]{0.5\textwidth}
\centering
\includegraphics[width=0.95\textwidth]{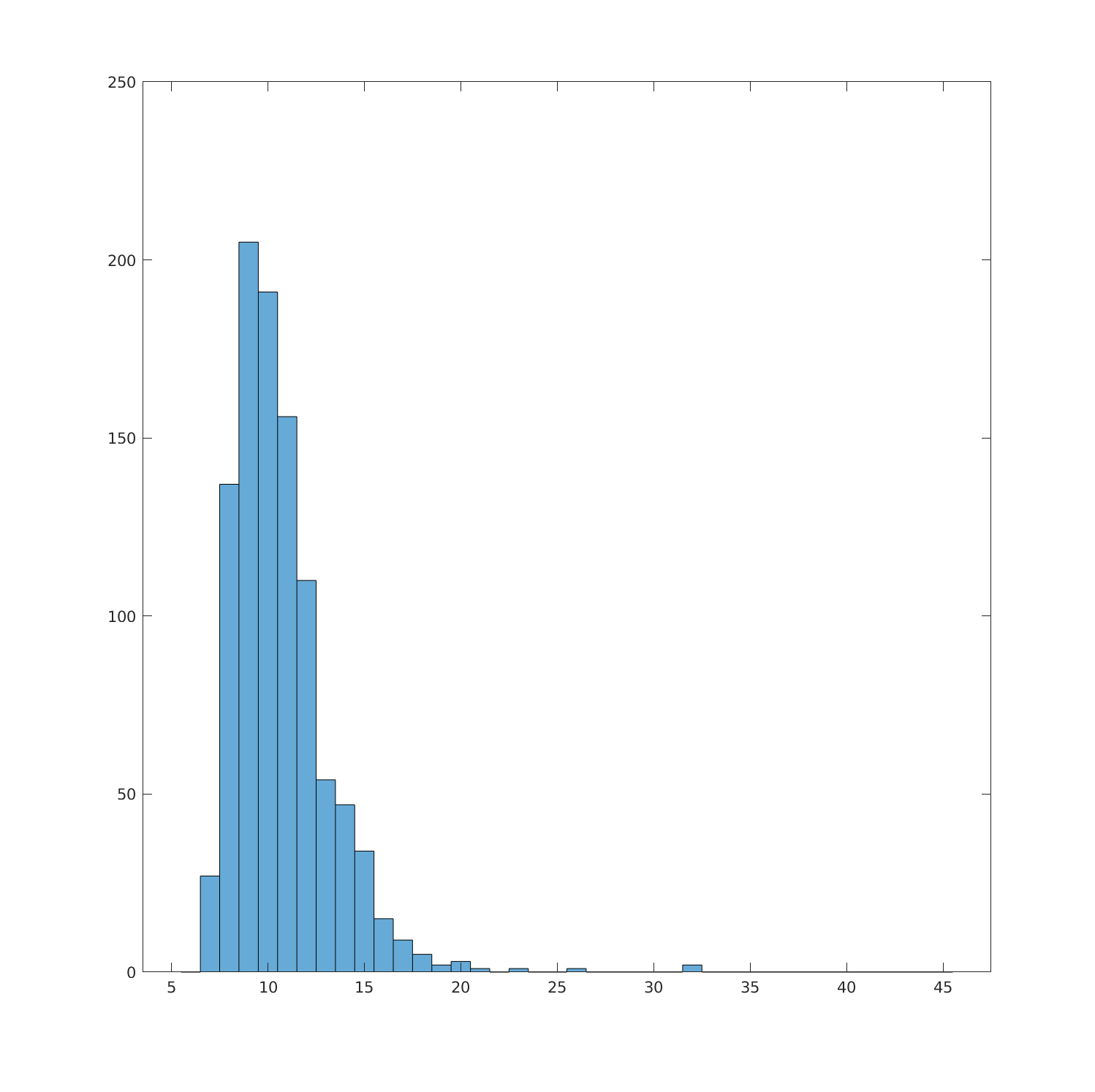}
\caption{$\nu=10$.}
\end{subfigure}
\begin{subfigure}[t]{0.5\textwidth}
\centering
\includegraphics[width=0.95\textwidth]{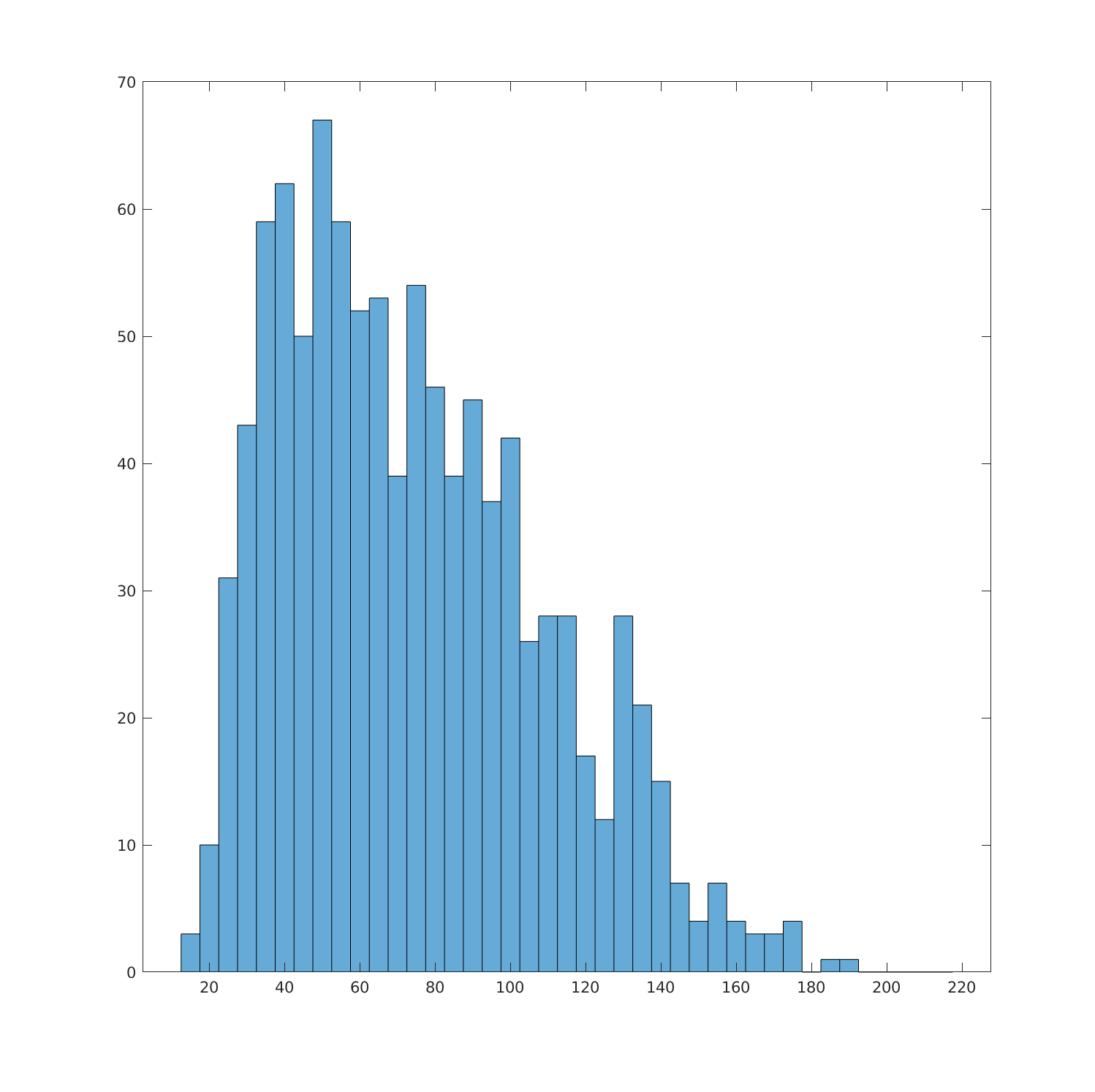}
\caption{$\nu=100$.}
\end{subfigure}\hfill
\begin{subfigure}[t]{0.5\textwidth}
\centering
\includegraphics[width=0.95\textwidth]{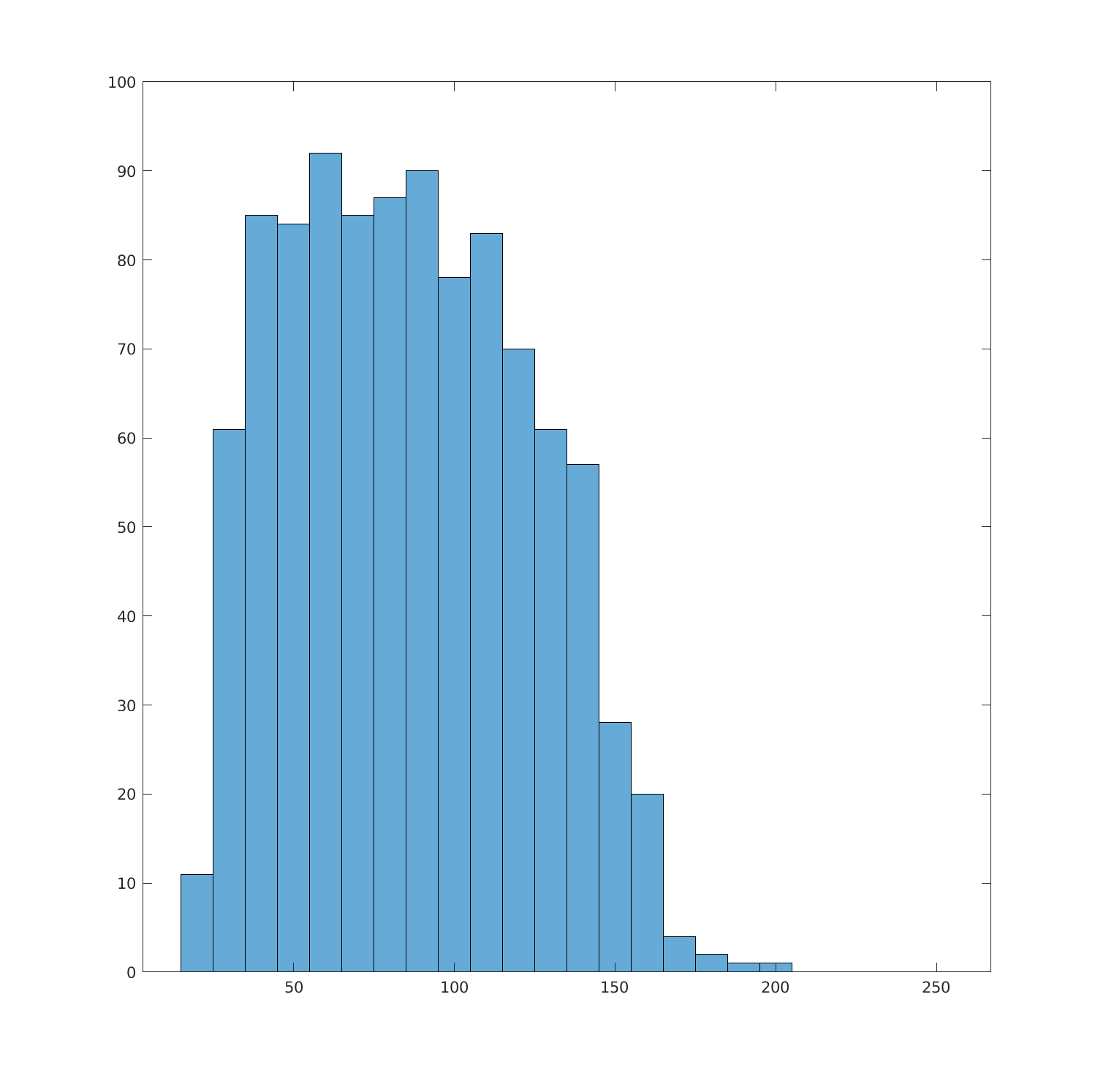}
\caption{$\nu=200$. }
\end{subfigure}\hfill
\caption{Histograms of the output $\nu$ from the algorithms.}
\label{fig:hist_plots}
\end{figure}

%------------------------------------------------------------------------------------
\subsection{Comparison with other Accelerations of the EM Algortihm}\label{sec:accel}
%------------------------------------------------------------------------------------

In this section, we compare our algorithms with the Expectation/Conditional Maximization Either (ECME) algorithm \cite{LR1994, LR95} and apply the SQUAREM acceleration \cite{VR2008} as well as the damped Anderson Acceleration (DAAREM) \cite{HV2019} to our algorithms.

\paragraph{ECME algorithm:}
The ECME algorithm was first proposed in \cite{LR1994}. 
Some numerical examples of the behavior of the ECME algorithm for estimating the parameters $(\nu,\mu,\Sigma)$ of a Student-$t$ distribution $T_\nu(\mu,\Sigma)$ are given in \cite{LR95}. 
The idea of ECME is first to replace the M-Step of the EM algorithm by the following update of the parameters $(\nu_r,\mu_r,\Sigma_r)$: 
first, we fix $\nu=\nu_r$ and compute the update $(\mu_{r+1},\Sigma_{r+1})$ 
of the parameters $(\mu_r,\Sigma_r)$ by performing one step of the EM algorithm for fixed degree of freedom (CM1-Step). 
Second, we fix $(\mu,\Sigma)=(\mu_r,\Sigma_r)$ and compute the update $\nu_{r+1}$ of $\nu_r$ by maximizing the likelihood function with respect to $\nu$ (CM2-Step).
The resulting algorithm is given in Algorithm \ref{alg:ECME}. 
It is similar to the GMMF (Algorithm \ref{alg:GMMF}), but uses the $\Sigma$-update of the EM algorithm (Algorithm \ref{alg:EM}) instead of the $\Sigma$-update of the aEM algorithm (Algorithm \ref{alg:aEM}). 
The authors of \cite{LR1994}   showed a similar convergence result as for the EM algorithm. 
Alternatively, we could prove Theorem \ref{cor:likelihood_decreases} 
for the ECME algorithm analogously as for the GMMF algorithm.\\

%-----------------------------------------
\begin{algorithm}[!ht]
	\caption{ECME Algorithm (ECME)} \label{alg:ECME}
	\begin{algorithmic}
		\State \textbf{Input:} $x_1,\ldots,x_n\in \R^d$, $n \geq d+1$, $w \in \mathring \Delta_n$ 
		\State \textbf{Initialization:} 
		$\nu_0 = \eps>0$,  $\mu_0 =\frac{1}{n} \sum\limits_{i=1}^n x_i$, 
		$\Sigma_0 =\frac{1}{n}\sum\limits_{i=1}^n (x_i-\mu_0)(x_i-\mu_0)^\tT$
		\For{$r=0,\ldots$}
		\vspace{0.2cm}
		
		 \textbf{E-Step:} Compute the weights
		\begin{align*} 	
		\delta_{i,r} &=  (x_i-\mu_r)^\tT \Sigma_r^{-1} (x_i-\mu_r)\\
		\gamma_{i,r} &=    \frac{\nu_r + d}{ \nu_r + \delta_{i,r} }
		\end{align*}
		
		\hspace*{0.2cm} \textbf{CM1-Step:} Update the parameters
		\begin{align*}
		\mu_{r+1}    
		&=   
		\frac{ \sum\limits_{i=1}^{n} w_i \gamma_{i,r} x_i}{ \sum\limits_{i=1}^{n} w_i\gamma_{i,r} } 
		\\
		\Sigma_{r+1} 
		&=   
		\sum\limits_{i=1}^{n} w_i \gamma_{i,r} (x_i-\mu_{r+1})(x_i-\mu_{r+1})^\tT 
		\end{align*}
		
        \hspace*{0.2cm} \textbf{CM2-Step:} Update the parameter
        \begin{align*}
		\nu_{r+1}&=  \; \text{ zero of } 
		 \phi\left( \frac{\nu}{2} \right) 
			-\phi\left( \frac{\nu +d}{2} \right)
			+ \sum_{i=1}^n w_i \left( \frac{\nu + d}{\nu +  \delta_{i,r+1}}  - \log\left( \frac{\nu + d}{\nu +  \delta_{i,r+1}} \right) - 1 \right)	 
		\end{align*}
		\EndFor
	\end{algorithmic}
\end{algorithm}
%-------------------------------------------------------------------------------------------

Next, we consider two acceleration schemes of arbitrary fixed point algorithms $\vartheta_{r+1}=G(\vartheta_r)$. In our case $\vartheta\in\R^p$ is given by $(\nu,\mu,\Sigma)$ and $G$ is given by one step of Algorithm \ref{alg:EM}, \ref{alg:aEM}, \ref{alg:MMF}, \ref{alg:GMMF} or \ref{alg:ECME}.

\paragraph{SQUAREM Acceleration:}

The first acceleration scheme, called squared iterative methods (SQUAREM) was proposed in \cite{VR2008}. 
The idea of SQUAREM is to update the parameters $\vartheta_r=(\nu_r,\mu_r,\Sigma_r)$ in the following way:
we compute $\vartheta_{r,1}=G(\vartheta_r)$ and $\vartheta_{r,2}=G(\vartheta_{r,1})$. 
Then, we calculate $s=\vartheta_{r,1}-\vartheta_r$ and $v=(\vartheta_{r,2}-\vartheta_{r,1})-s$. 
Now we set $\vartheta'=\vartheta_r-2\alpha r+\alpha^2 v$ and define the update $\vartheta_{r+1}=G(\vartheta')$, where $\alpha$ is chosen as follows. 
First, we set $\alpha=\min(-\tfrac{\|r\|_2}{\|v\|_2},-1)$. Then we compute $\vartheta'$ as described before. 
If $L(\vartheta')<L(\vartheta_r)$, we keep our choice of $\alpha$. 
Otherwise we update $\alpha$ by $\alpha=\tfrac{\alpha-1}{2}$.
Note that this scheme terminates as long a $\vartheta_r$ is not a critical point of $L$ by the following argument:
it holds that $\vartheta_r+2r+v=\vartheta_{r,2}$, 
which implies that it holds that $\lim_{\alpha\to-1}L(\vartheta_r-2\alpha+\alpha^2v)=L(\vartheta_{r,2})\leq L(\vartheta_r)$ 
with equality if and only if $\vartheta_r$ is a critical point of $L$, 
since all our algorithms have the property that $L(\vartheta)\geq L(G(\vartheta))$ with equality if and only if $\vartheta$ is a critical point of $L$.
By construction this scheme ensures that the negative log-likelihood values of the iterates is decreasing.

\paragraph{Damped Anderson Acceleration with Restarts and $\epsilon$-Monotonicity (DAAREM):}

The DAAREM acceleration was proposed in \cite{HV2019}. 
It is based on the Anderson acceleration, which was introduced in \cite{A1965}.
As for the SQUAREM acceleration want to solve 
the fixed point equation $\vartheta=G(\vartheta)$ with $\vartheta=(\nu,\mu,\Sigma)$ using the iteration $\vartheta_{r+1}=G(\vartheta_r)$. 
We also use the equivalent formulation to solve $f(\vartheta)=0$, where $f(\vartheta)=G(\vartheta)-\vartheta$. 
For a fixed parameter $m\in\N_{>0}$, we define $m_r=\min(m,r)$. 
Then, one update of $\vartheta_r$ using the Anderson Acceleration is given by 
\begin{align}
\vartheta_{r+1}=&G(\vartheta_r)-\sum_{j=1}^{m_r} (G(\vartheta_{r-m_r+j})-G(\vartheta_{r-m_r+j-1}))\gamma_j^{(r)}\label{eq:AA_update}\\
=&\vartheta_r+f(\vartheta_r)-\sum_{j=1}^{m_r} ((\vartheta_{r-m_r+j}-\vartheta_{r-m_r+j-1})-(f(\vartheta_{r-m_r+j})-f(\vartheta_{r-m_r+j-1})))\gamma_j^{(r)},
\end{align}
with $\gamma^{(r)}=(\mathcal{F}_r^\tT\mathcal{F}_r)^{-1}\mathcal{F}_r^\tT f(\vartheta_r)$, 
where the columns of $\mathcal{F}_r\in\R^{p\times m_r}$ are given by $f(\vartheta_{r-m_r+j+1})-f(\vartheta_{r-m_r+j})$ for $j=0,...,m_r-1$. 
An equivalent formulation of update step \eqref{eq:AA_update} is given by
\begin{align}
\vartheta_{r+1}=\vartheta_r+f(\vartheta_r)-(\mathcal{X}_r+\mathcal{F}_r)\gamma^{(r)},
\end{align}
where the columns of $\mathcal{X}_r\in\R^{p\times m_r}$ are given by $\vartheta_{r-m_r+j+1}-\vartheta_{r-m_r+j}$ for $j=0,...,m_r-1$.
The Anderson acceleration can be viewed as a special case of a multisecant quasi-Newton procedure to solve $f(\vartheta)=0$. For more details we refer to \cite{FS2009, HV2019}.\\

The DAAREM acceleration modifies the Anderson acceleration in three points. 
The first modification is to restart the algorithm after $m$ steps. 
That is, to set $m_r=\min(m,c_r)$ instead of $m_r=\min(m,r)$, where $c_r\in\{1,...,m\}$ is defined by $c_r=r\,\mathrm{mod}\,m$. 
The second modification is to add damping term in the computation coefficients $\gamma^{(r)}$. 
This means, that $\gamma^{(r)}$ is given by $\gamma^{(r)}=(\mathcal{F}_r^\tT\mathcal{F}_r+\lambda_r I)^{-1}\mathcal{F}_r^\tT f(\vartheta_r)$ instead of $\gamma^{(r)}=(\mathcal{F}_r^\tT\mathcal{F})^{-1}\mathcal{F}_r^\tT f(\vartheta_r)$. 
The parameter $\lambda_r$ is chosen such that
\begin{align}
\|(\mathcal{F}_r^\tT\mathcal{F}_r+\lambda_r I)^{-1}\mathcal{F}_r^\tT f(\vartheta_r)\|_2^2=\delta_r\|(\mathcal{F}_r^\tT\mathcal{F}_r)^{-1}\mathcal{F}_r^\tT f(\vartheta_r)\|_2^2\label{eq:DAAREM_lambda}
\end{align}
for some damping parameters $\delta_r$. We initialize the $\delta_r$ by $\delta_1=\tfrac1{1+\alpha^{\kappa}}$ and decrease the exponent of $\alpha$ in each step by $1$ up to a minimum of $\kappa-D$ for some parameter $D\in\N_{>0}$.
The third modification is to enforce that for the negative log-likelihood function $L$ does not increase more than $\epsilon$ in one iteration step. 
To do this, we compute the update $\vartheta_{r+1}$ using the Anderson acceleration. If $L(\vartheta_{r+1})>L(\vartheta_r)+\epsilon$, we use our original fixed point algorithm in this step, i.e.\ we set $\vartheta_{r+1}=G(\vartheta_r)$.

We summarize the DAAREM acceleration in Algorithm \ref{alg:DAAREM}. In our numerical experiments we use for the parameters the values suggested by \cite{HV2019}, that is $\epsilon=0.01$, $\epsilon_c=0$, $\alpha=1.2$, $\kappa=25$, $D=2\kappa$ and $m=\min(\lceil\tfrac{p}2\rceil,10)$, where $p$ is the number of parameters in $\vartheta$. 

%-----------------------------------------
\begin{algorithm}[!ht]
	\caption{DAAREM acceleration} \label{alg:DAAREM}
	\begin{algorithmic}
    \State \textbf{Input:} Parameters $\epsilon\geq0$, $\epsilon_c\geq0$, $\alpha>1$, $\kappa\geq0$, $D\geq0$, $m\geq1$
	\State \textbf{Initialization:}  Initialize $\vartheta_0=(\nu_0,\mu_0,\Sigma_0)$ as in the corresponding fixed point algorithm.
    \State Set $c_1=1$, $s_1=0$, $\vartheta_1=\vartheta_0+f(\vartheta_0)$, $L^*=L(x_1)$.
    \For{r=1,2,...}
    \State Set $m_r=\min(m,c_r)$, $\delta_r=\tfrac1{1+\alpha^{\kappa-s_r}}$ and compute $f_r=f(\vartheta_r)$.
    \State Define the columns of $\mathcal{F}_r,\mathcal{X}_r\in\R^{p\times m_k}$ by $f_{r-m_r+j+1}-f_{r-m_r+j}$ and $\vartheta_{r-m_r+j+1}-\vartheta_{r-m_r+j}$ respectively, $j=0,...,m_r-1$.
    \State Define $\lambda_r$ by \eqref{eq:DAAREM_lambda} and set $\gamma^{(r)}=(\mathcal{F}_r^\tT\mathcal{F}_r+\lambda_r I)^{-1}\mathcal{F}_r^\tT f_r$.
    \State Set $t_{r+1}=\vartheta_r+f_r-(\mathcal{X}_r+\mathcal{F_r})\gamma^{(r)}$
    \If{$L(t_{r+1})\leq L(\vartheta_r)+\epsilon$}
    \State Set $\vartheta_{r+1}=t_{r+1}$ and $s_\text{new}=s_r+1$.
    \Else
    \State Set $\vartheta_{r+1}=\vartheta_r+f_r$ and $s_\text{new}=s_r$.
    \EndIf
    \If{$k\,\mathrm{mod}\,m=0$}
    \If{$L(\vartheta_{r+1})> L^*+\epsilon_c$}
    \State Set $s_\text{new}=\max\{s_\text{new}-m,-D\}$
    \EndIf
    \State Set $c_{k+1}=1$ and $L*=L(\vartheta_{k+1})$.
    \Else
    \State Set $c_{r+1}=c_r+1$.
    \EndIf
    \EndFor
    \end{algorithmic}
\end{algorithm}
%-------------------------------------------------------------------------------------------

\paragraph{Simulation Study:}

To compare the performance of all of these algorithms we perform again a Monte Carlo simulation. As in the previous section we draw $n=100$ i.i.d.~realizations of $T_\nu(\mu,\Sigma)$ with $\mu=0$, $\Sigma=0.1\,\mathrm{Id}$ and $\nu\in\{1,2,5,10,100\}$. Then, we use each of the Algorithms \ref{alg:EM}, \ref{alg:aEM}, \ref{alg:MMF}, \ref{alg:GMMF} and \ref{alg:ECME} to compute the ML-estimator $(\hat\nu,\hat\mu,\hat\Sigma)$. We use each of these algorithms with no acceleration, with SQUAREM acceleration and with DAAREM acceleration.\\
We use the same initialization and stopping criteria as in the previous section and repeat this experiment $N=1.000$ times. To quantify the performance of the algorithms, we count the number of iterations and measure the execution time. The results are given in Table \ref{tab:performance2}.

We observe that for nearly any choice of the parameters the performance of the GMMF is better than the performance of the ECME. For small $\nu$, the performance of the SQUAREM-aEM is also very good. On the other hand, for large $\nu$ the SQUAREM-GMMF behaves very well. Further, for any choice of $\nu$ the performance of the SQUAREM-MMF is close to the best algorithm.

\begin{table}[htp]
\begin{center}
\resizebox*{7cm}{!}{
\begin{sideways}
\begin{tabular}{c|c c c c c}
Algorithm&$\nu=1$&$\nu=2$&$\nu=5$&$\nu=10$&$\nu=100$\\\hline
EM&$62.24\pm2.47$&$46.20\pm1.84$&$50.14\pm11.01$&$122.45\pm30.81$&$530.72\pm\n89.11$\\
aEM&$23.39\pm0.75$&$26.46\pm1.08$&$49.60\pm\n7.55$&$117.21\pm30.74$&$527.77\pm\n89.92$\\
MMF&$22.13\pm0.73$&$21.51\pm0.96$&$25.12\pm\n2.63$&$\n38.17\pm\n4.47$&$\n53.98\pm\n\n7.06$\\
GMMF&$20.56\pm0.67$&$17.79\pm0.79$&$12.06\pm\n1.73$&$\n14.35\pm\n0.97$&$\n10.86\pm\n\n2.10$\\
ECME&$60.81\pm2.41$&$40.73\pm1.97$&$29.07\pm\n1.81$&$\n22.12\pm\n3.81$&$\n12.81\pm\n\n2.96$\\\hline
DAAREM-EM&$22.09\pm4.05$&$22.26\pm4.59$&$20.39\pm\n5.42$&$\n24.72\pm\n6.34$&$\n28.09\pm\n\n6.93$\\
DAAREM-aEM&$15.52\pm1.57$&$14.90\pm2.39$&$15.35\pm\n3.22$&$\n17.84\pm\n4.41$&$\n20.07\pm\n\n3.68$\\
DAAREM-MMF&$15.16\pm1.45$&$14.02\pm2.09$&$13.12\pm\n2.09$&$\n14.99\pm\n3.62$&$\n66.86\pm630.74$\\
DAAREM-GMMF&$14.11\pm1.04$&$12.81\pm1.46$&$\n9.61\pm\n1.27$&$\n\n9.84\pm\n1.46$&$\n10.15\pm\n\n2.10$\\
DAAREM-ECME&$22.69\pm4.71$&$19.15\pm3.50$&$17.06\pm\n3.33$&$\n16.89\pm\n3.75$&$\n12.35\pm\n\n3.90$\\\hline
SQUAREM-EM&$26.36\pm2.25$&$21.77\pm4.56$&$21.43\pm\n3.13$&$\n46.01\pm10.72$&$111.24\pm\n40.47$\\
SQUAREM-aEM&$15.32\pm0.98$&$14.86\pm0.86$&$22.87\pm\n2.26$&$\n43.57\pm\n8.29$&$\n38.56\pm\n35.35$\\
SQUAREM-MMF&$15.47\pm1.09$&$14.05\pm1.40$&$14.18\pm\n1.56$&$\n18.40\pm\n1.21$&$\n22.41\pm\n\n9.39$\\
SQUAREM-GMMF&$\mathbf{13.30\pm1.49}$&$\mathbf{11.99\pm0.16}$&$\mathbf{\n9.02\pm\n0.49}$&$\mathbf{\n\n8.90\pm\n0.80}$&$\mathbf{\n\n8.28\pm\n\n1.29}$\\
SQUAREM-ECME&$24.25\pm2.79$&$19.20\pm1.96$&$18.48\pm\n3.12$&$\n17.98\pm\n3.33$&$\n13.41\pm\n\n3.41$
\end{tabular}
\end{sideways}}\quad
\resizebox*{7cm}{!}{
\begin{sideways}
\begin{tabular}{c|c c c c c}
Algorithm&$\nu=1$&$\nu=2$&$\nu=5$&$\nu=10$&$\nu=100$\\\hline
EM&$0.00890\pm0.00163$&$0.00644\pm0.00074$&$0.00682\pm0.00158$&$0.01659\pm0.00432$&$0.07076\pm0.01350$\\
aEM&$0.00365\pm0.00056$&$0.00401\pm0.00049$&$0.00732\pm0.00128$&$0.01706\pm0.00465$&$0.07513\pm0.01416$\\
MMF&$0.00369\pm0.00075$&$0.00342\pm0.00039$&$0.00390\pm0.00052$&$0.00589\pm0.00085$&$0.00834\pm0.00151$\\
GMMF&$0.00763\pm0.00193$&$0.00540\pm0.00061$&$0.00355\pm0.00074$&$0.00551\pm0.00063$&$0.00599\pm0.00112$\\
ECME&$0.01998\pm0.00343$&$0.01214\pm0.00137$&$0.00927\pm0.00114$&$0.00801\pm0.00105$&$0.00684\pm0.00157$\\\hline
DAAREM-EM&$0.00728\pm0.00163$&$0.00726\pm0.00158$&$0.00652\pm0.00180$&$0.00796\pm0.00218$&$0.00905\pm0.00233$\\
DAAREM-aEM&$0.00554\pm0.00095$&$0.00519\pm0.00097$&$0.00530\pm0.00124$&$0.00613\pm0.00160$&$0.00687\pm0.00141$\\
DAAREM-MMF&$0.00553\pm0.00090$&$0.00500\pm0.00084$&$0.00463\pm0.00082$&$0.00529\pm0.00137$&$0.02410\pm0.22518$\\
DAAREM-GMMF&$0.00837\pm0.00185$&$0.00679\pm0.00091$&$0.00491\pm0.00081$&$0.00601\pm0.00086$&$0.00772\pm0.00201$\\
DAAREM-ECME&$0.01527\pm0.00351$&$0.01061\pm0.00175$&$0.00968\pm0.00171$&$0.00993\pm0.00189$&$0.00825\pm0.00207$\\\hline
SQUAREM-EM&$0.00456\pm0.00081$&$0.00372\pm0.00077$&$0.00375\pm0.00068$&$0.00831\pm0.00220$&$0.02299\pm0.00837$\\
SQUAREM-aEM&$\mathbf{0.00291\pm0.00050}$&$0.00269\pm0.00029$&$0.00441\pm0.00065$&$0.00913\pm0.00203$&$0.00795\pm0.00621$\\
SQUAREM-MMF&$0.00308\pm0.00059$&$\mathbf{0.00268\pm0.00035}$&$\mathbf{0.00270\pm0.00041}$&$\mathbf{0.00373\pm0.00041}$&$0.00474\pm0.00184$\\
SQUAREM-GMMF&$0.00569\pm0.00129$&$0.00400\pm0.00040$&$0.00304\pm0.00042$&$0.00375\pm0.00046$&$\mathbf{0.00420\pm0.00080}$\\
SQUAREM-ECME&$0.01153\pm0.00222$&$0.00722\pm0.00086$&$0.00717\pm0.00112$&$0.00761\pm0.00090$&$0.00727\pm0.00182$
\end{tabular}
\end{sideways}}
\end{center}
\caption{Average number of iterations (top) and execution times (bottom) and the corresponding standard deviations of the different algorithms.}
\label{tab:performance2}
\end{table}

%------------------------------------------------------------------------------------
\subsection{Unsupervised Estimation of Noise Parameters} \label{sec:images}
%------------------------------------------------------------------------------------
Next, we provide an application in image analysis. To this aim, we consider images corrupted by one-dimensional Student-$t$ noise
with $\mu=0$ and unknown $\Sigma \equiv \sigma^2$ and $\nu$.
We provide a method 
that allows to estimate $\nu$ and $\sigma$ in an unsupervised way. 
The basic idea is to consider constant areas of an image, 
where the signal to noise ratio is weak and differences between pixel values 
are solely caused by the noise. 

\paragraph{Constant area detection:}
In order to detect constant regions in an image, we adopt an idea presented in~\cite{SDA15}. 
It is based on Kendall's $\tau$-coefficient, which is a measure of rank correlation, 
and the associated $z$-score, see~\cite{Ken38,Ken45}.
In the following, we briefly summarize the main ideas behind this approach.
 For finding  constant regions we proceed as follows: First, the image grid $\GG$ is partitioned into $K$ small, 
non-overlapping regions $\GG= \bigcup_{k=1}^K R_k$, and for each region we consider the hypothesis testing problem 
\begin{align}
\HH_0&\colon R_k\text{ is constant}\qquad \text{vs.}\qquad 
\HH_1\colon R_k\text{ is not constant}	\label{constant_test}.
\end{align}
To decide whether to reject $\HH_0$ or not, we observe the following: Consider a fixed region $R_k$ 
and let $I, J\subseteq R_k$ be two disjoint subsets of $R_k$ with the same cardinality. Denote with $u_I$ and $u_J$ the vectors 
containing the values of $u$ at the positions indexed by $I$ and $J$. Then, under $\HH_0$, the vectors $u_I$ and $u_J$ 
are uncorrelated (in fact even independent) for all choices of $I, J\subseteq R_k$ with $I\cap J = \emptyset$ and $|I|=|J|$. 
As a consequence, the rejection of $\HH_0$ can be reformulated as the question whether we can find  $I,J$ such that $u_I$ and $u_J$ 
are significantly correlated, since in this case there has to be some structure in the image region $R_k$ 
and it cannot be constant. Now, in order to quantify  the correlation, we adopt an idea presented in~\cite{SDA15} and make use of Kendall's $\tau$-coefficient, 
which is a measure of rank correlation, and the associated $z$-score, see~\cite{Ken38,Ken45}. 
The key
idea is to focus on the rank (i.e., on the relative order) of the values
rather than on the values themselves. In this vein, a block is considered homogeneous if
the ranking of the pixel values is uniformly distributed, regardless of the spatial arrangement of the pixels.  
In the following, we assume that we have extracted two disjoint subsequences $x = u_I$ and $y = u_J$ 
from a region $R_k$ with $I$ and $J$ as above.
Let $(x_i,y_i)$ and $(x_j,y_j)$ be two pairs of observations. Then, the pairs are said to be 
\begin{equation*}
\begin{cases}
\text{concordant} & \text{if } x_i<x_j \text{ and } y_i<y_j\\& \text{or } 
x_i>x_j \text{ and } y_i>y_j,\\
\text{discordant} & \text{if } x_i<x_j \text{ and } y_i>y_j\\& \text{or } x_i>x_j \text{ and } y_i<y_j,\\
\text{tied} & \text{if } x_i=x_j  \text{ or } y_i=y_j.
\end{cases}
\end{equation*}	 
Next, let $x,y\in \R^n$ be two sequences without tied pairs and let $n_c$ and $n_d$ be the number of concordant and discordant pairs, respectively. 
Then, \emph{Kendall's $\tau$ coefficient}~\cite{Ken38}	is defined as $\tau\colon \R^n\times \R^n\to [-1,1]$, 
\begin{equation*}
\tau(x,y) = \frac{n_c - n_d}{\frac{n(n-1)}{2}}. %=\frac{1}{n(n-1)}\sum_{1\leq i,j\leq n} \sign(x_i-x_j) \sign(y_i-y_j).
\end{equation*}
From this definition we see that if the agreement between the two rankings is perfect, i.e.\ the two rankings are the same, 
then the coefficient attains its maximal value 1. On the other extreme, if the disagreement between the two rankings is perfect, 
that is, one ranking is the reverse of the other, then the coefficient has value -1.
If the sequences $x$ and $y$ are uncorrelated,  we expect the coefficient to be approximately zero. Denoting with $X$ and $Y$ 
the underlying random variables that generated the sequences $x$ and $y$, we have the following result, 
whose proof can be found in~\cite{Ken38}. 

%--------------------------------
\begin{Theorem}\label{Theo:tau_asymptotic}
	Let $X$ and $Y$ be two arbitrary sequences under $\HH_0$ without tied pairs.
	Then, 	the random variable $\tau(X,Y)$ has an expected value of 0 and a variance of $\frac{2(2n+5)}{9n(n-1)}$. 
	Moreover, for $n\to \infty$, the associated \emph{$z$-score} $z\colon \R^n\times \R^n\to \R$,
	\begin{align*}
	z(x,y) = \frac{3\sqrt{n(n-1)}}{\sqrt{2(2n+5)}}\tau(x,y)=\frac{3\sqrt{2}(n_c - n_d)}{\sqrt{n(n-1)(2n+5)}} %=\frac{1}{n(n-1)}\sum_{1\leq i,j\leq n} \sign(x_i-x_j) \sign(y_i-y_j).
	\end{align*}
	is asymptotically standard normal distributed,
	\begin{equation*}
	z(X,Y)\overset{n\to \infty}{\sim}\NN(0,1).
	\end{equation*}
\end{Theorem}
%-------------------------------

With slight adaption, Kendall's $\tau$ coefficient can be generalized to sequences with tied pairs, see~\cite{Ken45}. 
As a consequence of Theorem~\ref{Theo:tau_asymptotic}, for a given significance level $\alpha\in (0,1)$, 
we can use the quantiles of the standard normal distribution to decide whether to reject $\HH_0$ or not. 
In practice, we cannot test any kind of region and any kind of disjoint sequences. 
As in~\cite{SDA15}, we restrict our attention to quadratic regions and pairwise comparisons of neighboring pixels. 
We use four kinds of neighboring relations (horizontal, 
vertical and two diagonal neighbors) thus perform in total four tests. We reject the hypothesis $\HH_0$ 
that the region is constant as soon as one of the four tests rejects it. 
Note that by doing so, the final significance level is smaller than the initially chosen one. 
We start with blocks of size $64\times 64$  
whose side-length is incrementally decreased until enough constant areas are found. 
\\

\paragraph{Parameter estimation.}
In each constant region we consider the pixel values in the region as i.i.d.\ 
samples of a univariate Student-$t$ distribution $T_\nu(\mu,\sigma^2)$,
where we estimate the  parameters  using  Algorithm~\ref{alg:MMF}. 

After estimating the parameters in each found constant region, 
the estimated location parameters  $\mu$   are discarded, 
while the estimated scale and degrees of freedom parameters $\sigma$ respective $\nu$  
are averaged to obtain the final estimate of the global noise parameters. 
At this point, as both $\nu$ and $\sigma$ influence the resulting distribution in a multiplicative way,  
instead of an  arithmetic mean, one might use a geometric which  is slightly less affected by outliers.

In Figure~\ref{Fig:constant_area} we illustrate this procedure for two different noise scenarios. 
The left column in each figure depicts  the detected constant areas. 
The middle and right column show histograms of the estimated values for $\nu$ respective $\sigma$.
For the constant area detection we use the code of~\cite{SDA15}\footnote{\url{https://github.com/csutour/RNLF}}.
The  true parameters used to generate the noisy images where $\nu=1$ and $\sigma = 10$ for the top row and $\nu=5$ and $\sigma = 10$ 
for the bottom row, while the obtained estimates are (geometric mean in brackets)
$\hat{\nu} = 1.0437$ ($1.0291$) 
and
$\hat{\sigma}= 10.3845$ ($10.3111$) for the top  row and $\hat{\nu}= 5.4140$ ($5.0423$) and $\hat{\sigma}=10.5500$ ($10.1897$) for the bottom row.

\begin{figure}[htp]
	\centering
	\begin{subfigure}[t]{0.32\textwidth}
		\centering
		\includegraphics[width=0.98\textwidth]{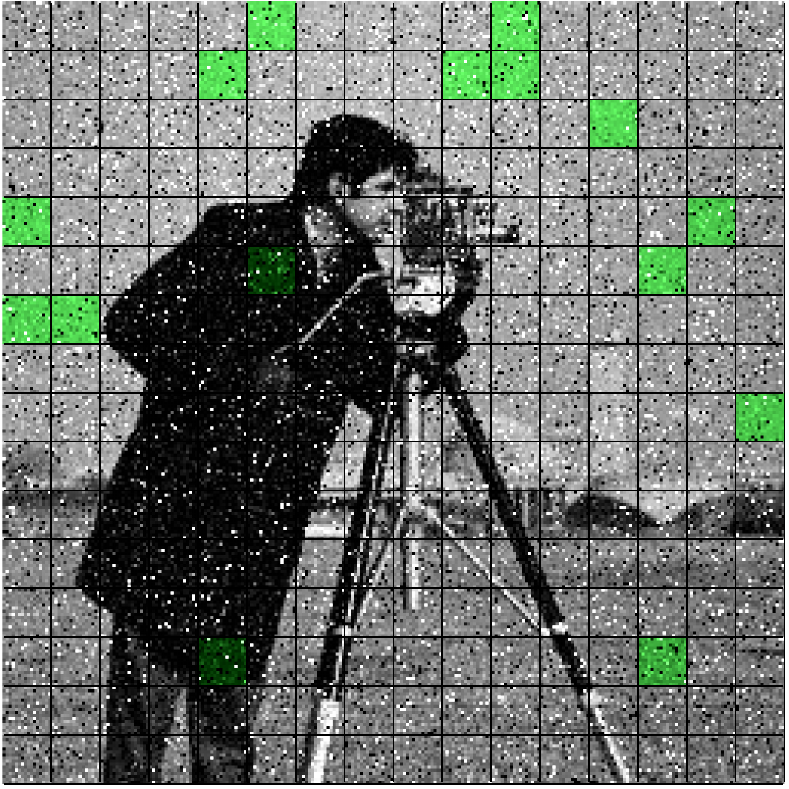}	
		\caption{Noisy image with detected homogeneous areas.}	
	\end{subfigure}
	\begin{subfigure}[t]{0.32\textwidth}
		\centering
		\includegraphics[width=\textwidth]{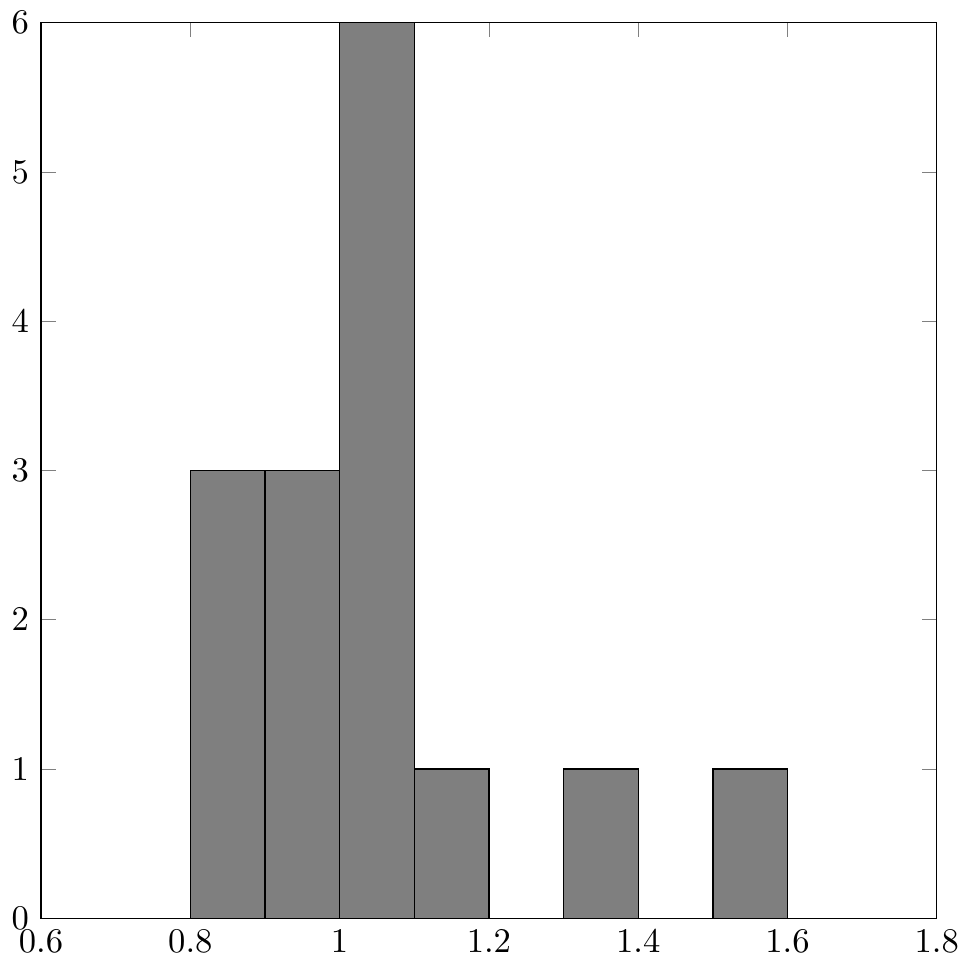}
		\caption{Histogram of estimates for $\nu$.}		
	\end{subfigure}
		\begin{subfigure}[t]{0.32\textwidth}
			\centering
			\includegraphics[width=\textwidth]{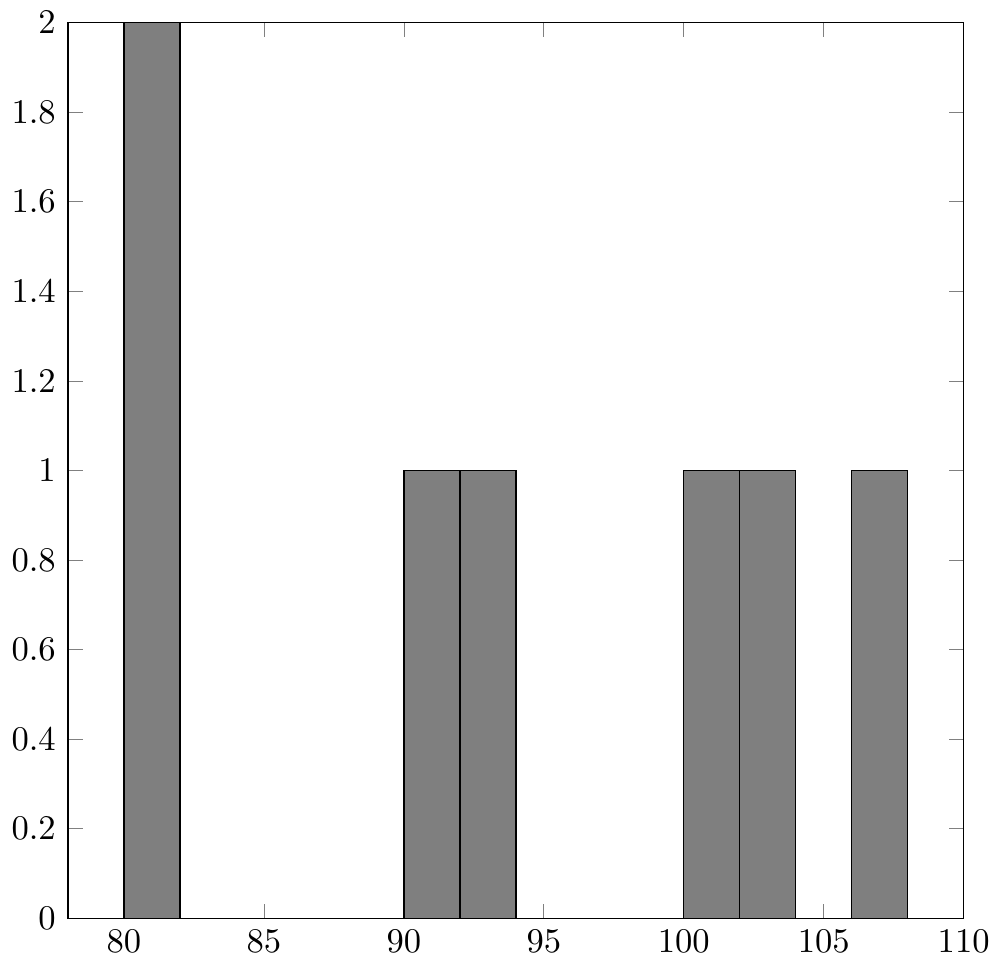}	
				\caption{Histogram of estimates for $\sigma^2$.}		
		\end{subfigure}
	
	\begin{subfigure}[t]{0.32\textwidth}
		\centering
		\includegraphics[width=0.98\textwidth]{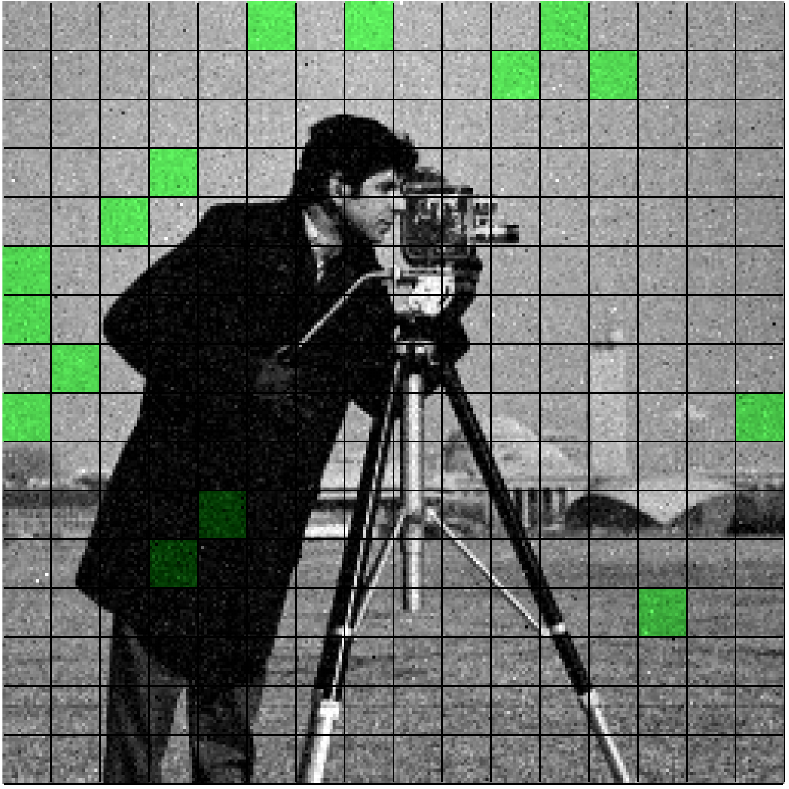}
			\caption{Noisy image with detected homogeneous areas.}	
	\end{subfigure}
	\begin{subfigure}[t]{0.32\textwidth}
		\centering
		\includegraphics[width=\textwidth]{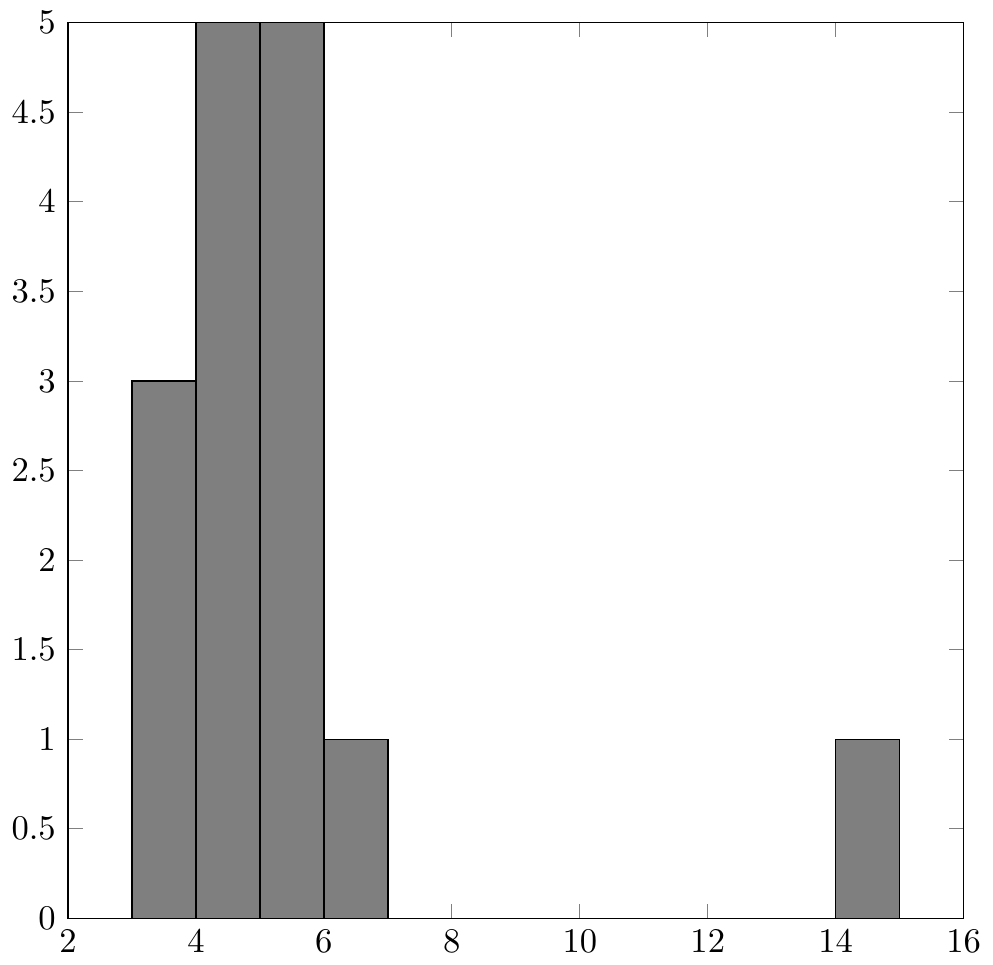}
		\caption{Histogram of estimates for $\nu$.}			
	\end{subfigure}
	\begin{subfigure}[t]{0.32\textwidth}
		\centering
		\includegraphics[width=\textwidth]{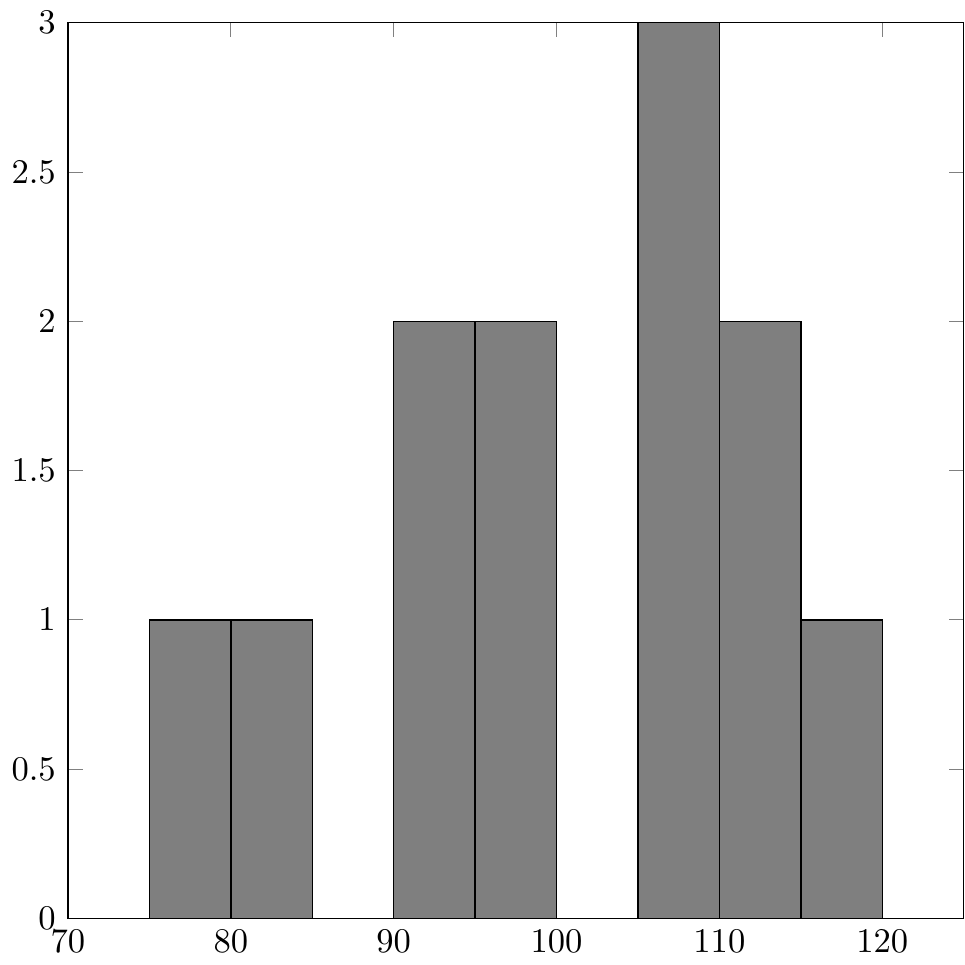}	
			\caption{Histogram of estimates for $\sigma^2$.}	
	\end{subfigure}
	\caption[]{Unsupervised estimation of the noise parameters $\nu$ and $\sigma^2$.  }\label{Fig:constant_area}
\end{figure}

A further example is given in Figure \ref{Fig:constant_area_muehle}. Here, the obtained estimates are (geometric mean in brackets)
$\hat{\nu} = 1.0075$ ($0.99799$) 
and
$\hat{\sigma}= 10.2969$ ($10.1508$) for the top  row and $\hat{\nu}= 5.4184$ ($5.1255$) and $\hat{\sigma}=10.2295$ ($10.1669$) for the bottom row.

\begin{figure}[htp]
	\centering
	\begin{subfigure}[t]{0.32\textwidth}
		\centering
		\includegraphics[height=4.7cm]{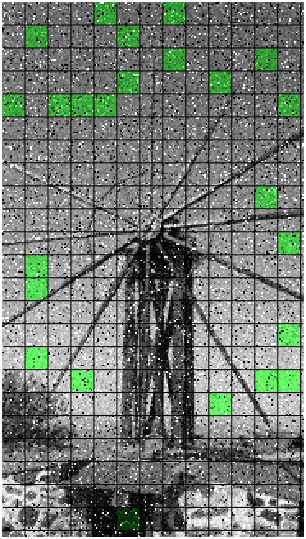}	
		\caption*{Noisy image with detected homogeneous areas.}	
	\end{subfigure}
	\begin{subfigure}[t]{0.32\textwidth}
		\centering
		\includegraphics[height=4.7cm]{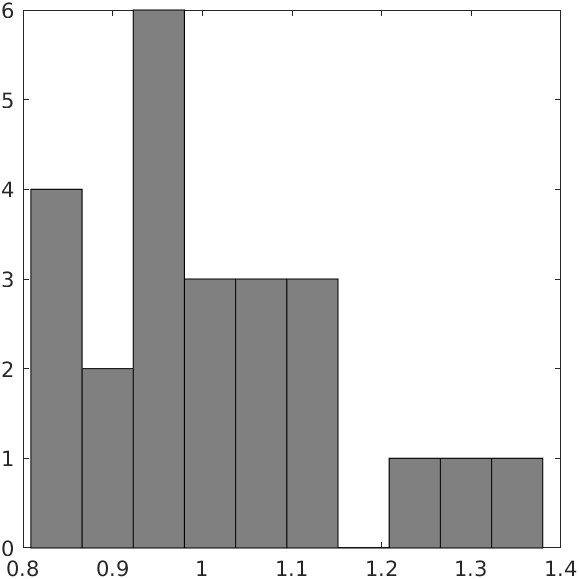}
		\caption*{Histogram of estimates for $\nu$.}		
	\end{subfigure}
		\begin{subfigure}[t]{0.32\textwidth}
			\centering
			\includegraphics[height=4.7cm]{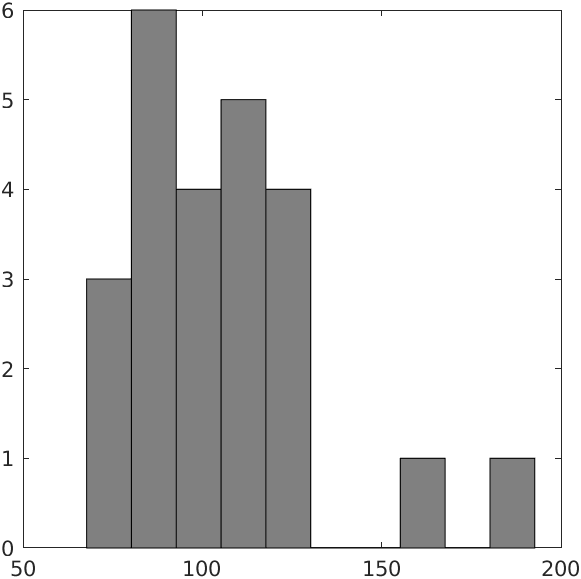}	
				\caption*{Histogram of estimates for $\sigma^2$.}		
		\end{subfigure}
	
	\begin{subfigure}[t]{0.32\textwidth}
		\centering
		\includegraphics[height=4.7cm]{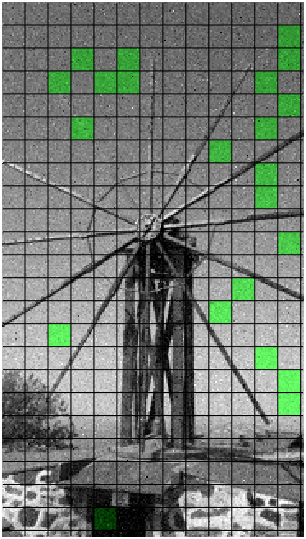}
			\caption*{Noisy image with detected homogeneous areas.}	
	\end{subfigure}
	\begin{subfigure}[t]{0.32\textwidth}
		\centering
		\includegraphics[height=4.7cm]{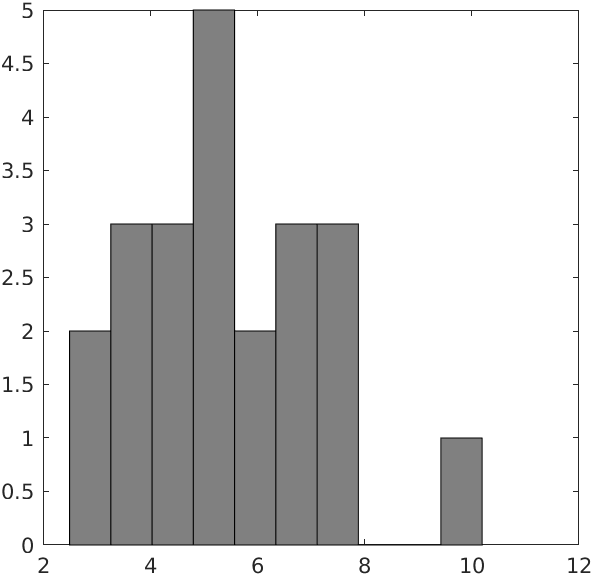}
		\caption*{Histogram of estimates for $\nu$.}			
	\end{subfigure}
	\begin{subfigure}[t]{0.32\textwidth}
		\centering
		\includegraphics[height=4.7cm]{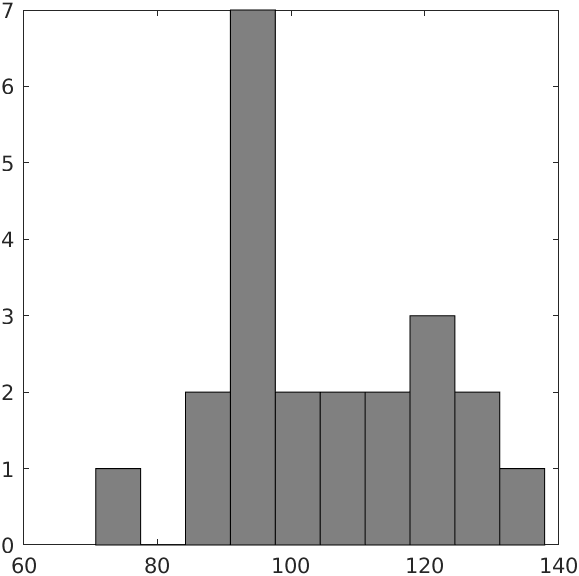}	
			\caption*{Histogram of estimates for $\sigma^2$.}	
	\end{subfigure}
	\caption[]{Unsupervised estimation of the noise parameters $\nu$ and $\sigma^2$.  }\label{Fig:constant_area_muehle}
\end{figure}

%------------------------------------------------
\appendix
\section{Auxiliary Lemmas}

\begin{Lemma}\label{lem:rike}
Let $x_i \in \mathbb R^d$, $i=1,\ldots,n$ and $w \in \mathring \Delta_n$ fulfill Assumption \ref{Ass:lin_ind}. 
Let $(\nu_r,\Sigma_r)_r$ be a sequence in $\R_{>0} \times \mathrm{SPD} (d)$ with $\nu_r \rightarrow 0$  as $r\rightarrow \infty$ (or if $\{\nu_r\}_r$ has a subsequence which converges to zero).
Then $(\nu_r,\Sigma_r)_r$ cannot be a minimizing sequence of $L(\nu,\Sigma)$.
\end{Lemma}

\begin{proof}
We write 
$$
L(\nu,\Sigma)=g(\nu)+L_\nu(\Sigma),
$$
where
$$
g(\nu)=2\log\left(\Gamma\left(\frac{\nu}{2}\right)\right)-2\log\left(\Gamma\left(\frac{d+\nu}{2}\right)\right)-\nu\log(\nu).
$$
Then it holds $\lim_{\nu\to0}g(\nu)=\infty$. Hence it is sufficient to show that $(\nu_r,\Sigma_r)_r$ has a subsequence $(\nu_{r_k},\Sigma_{r_k})$ 
such that $\left( L_{\nu_{r_k}}(\Sigma_{r_k}) \right)_r$ 
is bounded from below. Denote by $\lambda_{r1}\geq...\geq\lambda_{rd}$ the eigenvalues of $\Sigma_{r}$.
\\[1ex]
\textbf{Case 1:} Let $\{\lambda_{r,i}:r\in\mathbb N,i=1,\ldots,d\}\subseteq [a,b]$ for some $0<a\leq b<\infty$. 
Then it holds $\liminf_{r\to\infty}\log\abs{\Sigma_r}\geq \log(a^d)=d\log(a)$ and
$$
\liminf_{r\to\infty}(d+\nu_r)\sum_{i=1}^nw_i\log(\nu_r+x_i^\tT\Sigma_r^{-1}x_i)
\geq\lim_{r\to\infty}(d+\nu_r)\sum_{i=1}^nw_i\log\Bigl(\frac1b x_i^\tT x_i\Bigr)=d\sum_{i=1}^nw_i\log\Bigl(\frac1b x_i^\tT x_i\Bigr).
$$
Note that Assumption \ref{Ass:lin_ind} ensures $x_i\neq 0$ and $x_i^\tT x_i>0$ for $i=1,\ldots,n$. Then we get
\begin{align}
\liminf_{r\to\infty}L_{\nu_r}(\Sigma_r)
&=\liminf_{r\to\infty}(d+\nu_r)\sum_{i=1}^nw_i\log(\nu_r+x_i^\tT \Sigma_r^{-1}x_i)+\log\abs{\Sigma_r}\\
&\geq d\sum_{i=1}^nw_i\log\Bigl(\frac1b x_i^\tT x_i\Bigr)+d\log(a).
\end{align}
Hence $(L_{\nu_r}(\Sigma_r))_r$ is bounded from below and $(\nu_r,\Sigma_r)$ cannot be a minimizing sequence.
\\[1ex]
\textbf{Case 2:} Let $\{\lambda_{r,i}:r\in\mathbb N,i=1,\ldots,d\}\not\subseteq [a,b]$ for all $0<a\leq b<\infty$. 
Define $\rho_r=\|\Sigma_r\|_{F}$ and $P_r=\frac{\Sigma_r}{\rho_r}$. Then, by concavity of the logarithm, it holds 
\begin{align}
L_{\nu_r}(\Sigma_r)
&=(d+\nu_r)\sum_{i=1}^nw_i\log(\nu_r+x_i^\tT\Sigma_r^{-1}x_i)+\log(\abs{\Sigma_r})\\
&\geq d\sum_{i=1}^nw_i\log(x_i^\tT \Sigma_r^{-1}x_i)+ \nu_r\log(\nu_r)+\log(\abs{\Sigma_r})\\
&\geq d\sum_{i=1}^nw_i\log(\frac1{\rho_r}x_i^\tT P_r^{-1}x_i)+\log(\rho_r^d \abs{P_r})+\text{const}\\
&= \underbrace{d\sum_{i=1}^nw_i\log(x_i^\tT P_r^{-1}x_i)+\log(\abs{P_r})}_{\eqqcolon L_0(P_r)}+\text{const} \label{eq:sigma_to_p}.
\end{align}
Denote by $p_{r,1}\geq\ldots\geq p_{r,d}>0$ the eigenvalues of $P_r$. Since $\{P_r:r\in\mathbb N\}$ is bounded there exists some $C>0$ with $C\ge p_{r,1}$ for all $r\in\mathbb N$. 
Thus one of the following cases is fulfilled:
\begin{enumerate}
\item[i)] There exists a constant $c>0$ such that $p_{r,d}>c$ for all $r\in\mathbb N$. 
\item[ii)] There exists a subsequence $(P_{r_k})_k$ of $(P_r)_r$ which converges to some $P\in\partial\SPD(d)$.
\end{enumerate}

\textbf{Case 2i)} Let $c>0$ with $p_{r,d} \ge c$ for all $r\in\mathbb N$. 
Then  $\liminf\limits_{r\to\infty} \log(\abs{P_r}) \geq \log(c^d)=d\log(c)$ and
$$
\liminf_{r\to\infty}d\sum_{i=1}^nw_i\log(x_i^\tT P_r^{-1} x_i)\geq d\sum_{i=1}^nw_i\log\Bigl(\frac1C x_i^\tT x_i\Bigr).
$$
By \eqref{eq:sigma_to_p} this yields
\begin{align}
\liminf_{r\to\infty}L_{\nu_r}(\Sigma_r)
&\geq \liminf_{r\to\infty}d\sum_{i=1}^n w_i\log(x_i^\tT P_r^{-1}x_i)+\log(\abs{P_r})+\text{const}\\
&\geq d\sum_{i=1}^n w_i\log\Bigl(\frac1C x_i^\tT x_i\Bigr)+d\log(c)+\text{const}.
\end{align}
Hence $(L_{\nu_r}(\Sigma_r))_r$ is bounded from below and $(\nu_r,\Sigma_r)$ cannot be a minimizing sequence.
\\[1ex]
\textbf{Case 2ii)} We use similar arguments as in the proof of \cite[Theorem 4.3]{LS2019}. 
Let $(P_{r_k})_k$ be a subsequence of $(P_r)_r$ 
which converges to some $P\in\partial\SPD(d)$.
For simplicity we denote $(P_{r_k})_k$ again by $(P_r)_r$. Let $p_1\geq\ldots\geq p_d\geq0$ be the eigenvalues of $P$. 
Since $\|P\|_{F}=\lim_{r\to\infty} \| P_r\|_{F}=1$ it holds $p_1>0$. 
Let $q\in 1,\ldots,d-1$ such that $p_1\geq\ldots\geq p_q>p_{q+1}=\ldots=p_d=0$.
By $e_{r,1},\ldots,e_{,rd}$ we denote the orthonormal eigenvectors corresponding to $p_{r,1},\ldots,p_{r,d}$. 
Since $(\mathbb S^d)^d$ is compact we can assume (by going over to a subsequence) that $(e_{r,1},\ldots,e_{r,d})_r$ 
converges to orthonormal vectors $(e_1,\ldots,e_d)$. Define $S_0\coloneqq\{0\}$ and for $k=1,\ldots,d$ set 
$S_k\coloneqq \mathrm{span}\{e_1,\ldots,e_k\}$. 
Now, for $k=1,\ldots,d$ define
$$
W_k\coloneqq S_k\backslash S_{k-1}=\{y\in\R^d:\inner{y}{e_k}\neq 0, \inner{y}{e_l}=0 \text{ for }l=k+1,\ldots,d\}.
$$
Further, let
$$
\tilde I_k\coloneqq\{i\in\{1,\ldots,n\}:x_i\in S_k\}\quad\text{and}\quad I_k\coloneqq\{i\in\{1,\ldots,n\}:x_i\in W_k\}.
$$
Because of $S_k=W_k\dot\cup S_{k-1}$ we have $\tilde I_k=I_k\dot\cup\tilde I_{k-1}$ for $k=1,\ldots,d$. 
Due to Assumption \ref{Ass:lin_ind} we have $\abs{I_k}\leq\abs{\tilde I_k}\leq \dim(S_k)=k$ for $k=1,\ldots,d-1$. 
Defining for $j=1,\ldots,d$,
$$
L_j(P_r)\coloneqq d\sum_{i\in I_j} w_i\log(x_i^\tT P_r^{-1} x_i)+\log(p_{rj}),
$$
it holds $L_0(P_r)=\sum_{j=1}^d L_j$. For $j\leq q$ we get
$$
\liminf_{r\to\infty} L_j(P_r)
\geq 
\liminf_{r\to\infty}d\sum_{i\in I_j} w_i\log\Bigl(\frac1C x_i^\tT x_i\Bigr)+\log(p_{r,j})
=d\sum_{i\in I_j}w_i\log\Bigl(\frac1C x_i^\tT x_i\Bigr)+\log(p_j). 
$$
Since for $k\in\{1,\ldots,d\}$ and $i\in I_k$,
$$
x_i^\tT P_r^{-1} x_i=\sum_{j=1}^d\frac1{p_{r,j}}\inner{x_i}{e_{r,j}}^2\geq\frac1{p_{r,k}}\inner{x_i}{e_{rk}}^2,
$$
and $\lim_{r\to\infty}\inner{x_i}{e_{rk}} = \inner{x_i}{e_k}\neq0$, we obtain
$$
\liminf_{r\to\infty}p_{r,k}x_i^\tT P_r x_i
\geq\liminf_{r\to\infty}\inner{y}{e_{r,k}}
\geq \inner{y}{e_k}^2>0.
$$
Hence it holds for $j\ge q+1$ that
\begin{align}
L_j(P_r)
&=
d\sum_{i\in I_j}w_i(\log(x_i^\tT P_r^{-1}x_i)+\log(p_{r,j}))+\Big(1-d\sum_{i\in I_j}w_i\Big)\log(p_{r,j})\\
&= 
d\sum_{i\in I_j}w_i\log(p_{r,j}x_i^\tT P_r^{-1}x_i)+\Big(1-d\sum_{i\in I_j}w_i\Big)\log(p_{r,j}).
\end{align}
Thus we conclude
\begin{align}
\liminf_{r\to\infty}L_0(P_r)
&=
\liminf_{r\to\infty}\sum_{j=1}^d L_j(P_r)\geq\sum_{j=1}^q\liminf_{r\to\infty}L_j(P_r) + \liminf_{r\to\infty}\sum_{j=q+1}^d L_j(P_r)
\\
&\geq
\sum_{j=1}^q d\sum_{i\in I_j}w_i\log(\frac1Cx_i^\tT x_i)+\log(p_j)+\liminf_{r\to\infty}\sum_{j=q+1}^d d\sum_{i\in I_j}w_i\log(p_{rj}x_i^\tT P_r^{-1}x_i)\\
&\;\; +\liminf_{r\to\infty}\sum_{j=q+1}^d\Big(1-d\sum_{i\in I_j}w_i\Big)\log(p_{rj})\\
&\geq
\sum_{j=1}^qd\sum_{i\in I_j}w_i\log(\frac1Cx_i^\tT x_i)+\log(p_j)
+\sum_{j=q+1}^d d\sum_{i\in I_j }w_i\log(\inner{x_i}{e_j})\\
&\; \; +\liminf_{r\to\infty}\sum_{j=q+1}^d \Big(1-d\sum_{i\in I_j }w_i\Big)\log(p_{r,j})\\
&=
\text{const}+\liminf_{r\to\infty}\sum_{j=q+1}^d\Big(1-d\sum_{i\in I_j}w_i\Big)\log(p_{r,j}).
\end{align}
It remains to show that there exist $\tilde c > 0$ such that
\begin{equation}\label{eq:liminf}
\liminf_{r\to\infty}\sum_{j=q+1}^d\Big(1-d\sum_{i\in I_j}w_i\Big)\log(p_{r,j}) \ge \tilde c.
\end{equation}
We prove for $k\geq q+1$ by induction that for sufficiently large $r\in \mathbb N$ it holds
\begin{equation}\label{eq:induction}
\sum_{j=k}^d \Big(1-d\sum_{i\in I_j}w_i \Big) \log(p_{rj})\geq \Big(d \sum_{i\in\tilde I_{k-1}} w_i -(k-1)\Big)\log(p_{r,k}).
\end{equation}
\textbf{Induction basis $k=d$:}
Since $\tilde I_k=I_k\cup \tilde I_{k-1}$ we have
$$
\sum_{i\in\tilde I_k} w_i-\sum_{i\in\tilde I_{k-1}}w_i=\sum_{i\in I_k}w_i,
$$
and further
$$
1-d\sum_{i\in I_d} w_i
=1 - d\Big(\sum_{i\in\tilde I_d} w_i - \sum_{i\in\tilde I_{d-1}} w_i\Big)
=1-d \Big(1-\sum_{i\in\tilde I_{d-1}} w_i\Big)
=d\sum_{i\in\tilde I_{d-1}} w_i-(d-1).  
$$
If we multiply both sides with $\log(p_{rd})$ this yields \eqref{eq:induction} for $k=d$.
\\
\textbf{Induction step:} Assume that \eqref{eq:induction} holds for some $k+1$ with $d\geq k+1 >q+1$ i.e.
$$
\sum_{j=k+1}^d\Big(1-d\sum_{i\in I_j}w_i\Big)\log(p_{r,j})\geq d\Big(\sum_{i\in\tilde I_k}w_i-\frac{k}{d}\Big)\log(p_{r,k+1}).
$$
Then we obtain
\begin{align}
&\sum_{j=k}^d \Big(1-d\sum_{i\in I_j}w_i\Big)\log(p_{r,j})\\
&=\sum_{j=k+1}^d\Big(1-d\sum_{i\in I_j}w_i\Big)\log(p_{r,j})+\Big(1-d\sum_{i\in I_k} w_i\Big)\log(p_{r,k})\\
&\geq d\Big(\sum_{i\in\tilde I_k}w_i-\frac{k}{d}\Big)\log(p_{r,k+1})+\Big(1-d\sum_{i\in I_k} w_i\Big)\log(p_{r,k}).
\end{align}
and since $\sum_{i\in\tilde I_k}w_i<\abs{\tilde I_k}\frac1d\leq\frac{k}{d}$ by Assumption \ref{Ass:lin_ind} and $p_{r,k+1}\leq p_{r,k}<1$ 
finally
\begin{align}
&\geq d \Big(\sum_{i\in\tilde I_k}w_i - \frac{k}{d}\Big) \log(p_{r,k})+\Big(1-d\sum_{i\in I_k} w_i\Big) \log(p_{r,k})\\
&=\Big(d\sum_{i\in\tilde I_{k-1}}w_i-(k-1)\Big)\log(p_{r,k}).
\end{align}
This shows \eqref{eq:induction} for $k\geq q+1$. Using $k=q+1$ in \eqref{eq:liminf} we get 
\begin{equation}
\liminf_{r\to\infty}\sum_{j=q+1}^d\Big(1-d\sum_{i\in I_j}w_i\Big)\log(p_{rj})
\geq
\liminf_{r\to\infty}\underbrace{\Big(d\sum_{i\in\tilde I_{q}}w_i-q\Big)}_{<0}\underbrace{\log(p_{r,q+1})}_{\text{bounded from above}}>-\infty.
\end{equation}
This finishes the proof.
\end{proof}

\begin{Lemma}\label{lem:lik}
	Let $(\nu_r,\Sigma_r)_r$ be a sequence in $\R_{>0}\times\mathrm{SPD}(d)$ such that there exists $\nu_-\in\R_{>0}$ 
	with $\nu_-\leq\nu_r$ for all $r\in\mathbb N$. Denote by $\lambda_{r,1}\geq\cdots\geq\lambda_{r,d}$ the eigenvalues of $\Sigma_r$.
	If $\{\lambda_{1,r}:r\in\mathbb N\}$ is unbounded or $\{\lambda_{d,r}:r\in\mathbb N\}$ has zero as a cluster point, 
	then there exists a subsequence $(\nu_{r_k},\Sigma_{r_k})_k$ of $(\nu_r,\Sigma_r)_r$, such that 
	$\lim\limits_{k\to\infty}L(\nu_{r_k},\Sigma_{r_k})=\infty$.
\end{Lemma}

\begin{proof}
	Without loss of generality we assume (by considering a subsequence) 
	that either $\lambda_{r1}\to\infty$ as $r\to\infty$ and $\lambda_{rd}\geq c>0$ for all $r\in\mathbb N$ or that $\lambda_{rd}\to0$ as $r\to\infty$. 
	By \cite[Theorem 4.3]{LS2019} for fixed $\nu = \nu_-$, we have  $L_{\nu_-}(\Sigma_r) \to\infty$ as $r\to\infty$.
	
	The function $h\colon \R_{>0}\to \R$ defined by $\nu\mapsto (d+\nu)\log(\nu+k)$ is monotone increasing for all $k\in\R_{\geq 0}$. 
	This can be seen as follows:  The derivative of $h$ fulfills
	$$
	h'(\nu)=\frac{d+\nu}{k+\nu}+\log(\nu+k)\geq \frac{1+\nu}{k+\nu}+\log(\nu+k)
	$$
	and since
	$$
	\frac{\partial}{\partial k} \left(\frac{1+\nu}{k+\nu}+\log(\nu+k)\right)=\frac{k-1}{(k+\nu)^2}
	$$
	the later function is minimal for $k=1$, so that
	$$
	h'(\nu)\geq \frac{1+\nu}{k+\nu}+\log(\nu+k)\geq\frac{1+\nu}{1+\nu}+\log(\nu+1)=1+\log(1+\nu)>0.
	$$
	Using this relation, we obtain
	$$
	(d+\nu_r)\sum_{i=1}^n w_i\log\left(\nu_r+x_i^\tT\Sigma_{r}^{-1}x_i\right)
	\geq
	(d+\nu_-)\sum_{i=1}^n w_i\log\left(\nu_-+x_i^\tT\Sigma_{r}^{-1}x_i\right)
	$$
	and further
	\begin{align}
	L(\nu_r,\Sigma_r)
	&= 
	(d+\nu_r)\sum_{i=1}^n w_i\log\left(\nu_r+x_i^\tT\Sigma_{r}^{-1}x_i\right)+\log(\abs{\Sigma_r})\\
	&\geq
	(d+\nu_-)\sum_{i=1}^n w_i\log\left(\nu_-+x_i^\tT\Sigma_r^{-1}x_i\right)+\log(\abs{\Sigma_r})\\
	&=L_{\nu_-}(\Sigma_r)\to\infty \qquad \mathrm{as} \quad  r\to\infty.
	\end{align}
\end{proof}

%-----------------------------------------------------------
\section*{Acknowledgment}
The authors want to thank the anonymous referees for bringing certain accelerations of the EM algorithm 
to our attention. 
\\
 Funding by the German Research Foundation (DFG) with\-in the project STE 571/16-1 is gratefully acknowledged.

\bibliographystyle{abbrv}
\bibliography{Student_t}

\end{document}